\documentclass[11pt,reqno]{amsart}

\usepackage[margin=1in]{geometry}

\usepackage{hyperref}
\hypersetup{hypertex=true, colorlinks=true, linkcolor=black, anchorcolor=black, citecolor=black}

\usepackage{amsthm}
\usepackage{amsmath}
\numberwithin{equation}{section}
\usepackage{amssymb}
\usepackage{mathpazo}
\usepackage{mathrsfs}
\usepackage{mathtools}
\usepackage{libertine}

\usepackage[table, svgnames]{xcolor}

\newtheorem{theorem}{Theorem}[section]
\newtheorem{lemma}[theorem]{Lemma}
\newtheorem{corollary}[theorem]{Corollary}
\newtheorem{proposition}[theorem]{Proposition}

\theoremstyle{definition}
\newtheorem{definition}[theorem]{Definition}

\newtheorem{remark}[theorem]{Remark}

\newcommand{\PSH}{\textup{PSH}}

\newcommand{\p}{\partial}

\renewcommand{\bar}{\overline}
\newcommand{\eps}{\varepsilon}

\newcommand{\beq}{\begin{equation}}
\newcommand{\eeq}{\end{equation}}
\newcommand{\ee}{\end{eqnarray*}}
\newcommand{\be}{\begin{eqnarray*}}

\newcommand{\bd}{\begin{enumerate}}
	\newcommand{\ed}{\end{enumerate}}

\renewcommand{\hat}{\widehat}

\newcommand{\RR}{{\mathbb R}}

\newcommand\cE{{\mathcal E}}

\newcommand\cH{{\mathcal H}}

\newcommand\cL{{\mathcal L}}

\newcommand\cP{{\mathcal P}}

\newcommand{\vol}{\mathrm{vol}}

\newcommand{\Ent}{\mathrm{Ent}}

\newcommand{\tr}{\mathrm{tr}}

\newcommand{\Ric}{\mathrm{Ric}}
\renewcommand{\d}{\mathrm{d}}

\newcommand{\reg}{\mathrm{reg}}

\newcommand{\ddc}{\mathrm{dd}^{\rm c}}

\allowdisplaybreaks[4]

\makeatletter
\newcommand{\constant}[2]{%
  \@ifundefined{c@#1counter}{\newcounter{#1counter}}{}%
  \refstepcounter{#1counter}%
  \protected@write\@auxout{}{%
    \string\newlabel{#2}{{\arabic{#1counter}}{\thepage}{}{#1@\arabic{#1counter}}{}}%
  }%
  #1_{\arabic{#1counter}}%
}
\makeatother

\usepackage{seqsplit}

\definecolor{labelkeycolor}{named}{LightYellow}

\makeatletter
\renewcommand{\tocsection}[3]{%
  \indentlabel{\@ifnotempty{#2}{\bfseries\ignorespaces#1 #2\quad}}\bfseries#3}
\renewcommand{\tocsubsection}[3]{%
  \indentlabel{\@ifnotempty{#2}{\ignorespaces#1 #2\quad}}#3}

\newcommand\@dotsep{4.5}
\def\@tocline#1#2#3#4#5#6#7{\relax
  \ifnum #1>\c@tocdepth
  \else
    \par \addpenalty\@secpenalty\addvspace{#2}%
    \begingroup \hyphenpenalty\@M
    \@ifempty{#4}{%
      \@tempdima\csname r@tocindent\number#1\endcsname\relax
    }{%
      \@tempdima#4\relax
    }%
    \parindent\z@ \leftskip#3\relax \advance\leftskip\@tempdima\relax
    \rightskip\@pnumwidth plus1em \parfillskip-\@pnumwidth
    #5\leavevmode\hskip-\@tempdima{#6}\nobreak
    \leaders\hbox{$\m@th\mkern \@dotsep mu\hbox{.}\mkern \@dotsep mu$}\hfill
    \nobreak
    \hbox to\@pnumwidth{\@tocpagenum{\ifnum#1=1\bfseries\fi#7}}\par
    \nobreak
    \endgroup
  \fi}
\def\l@subsection{\@tocline{2}{0pt}{2.5pc}{5pc}{}}
\def\l@subsubsection{\@tocline{3}{0pt}{5pc}{5pc}{}}
\makeatother

\usepackage{fancyhdr}
\begin{document}

\title{The inverse Monge--Amp\`ere flow on klt Fano varieties}

\author{Jianchun Chu}
\address{School of Mathematical Sciences, Peking University, Yiheyuan Road 5, Beijing 100871, People's Republic of China}
\email{jianchunchu@math.pku.edu.cn}

\author{Tam\'as Darvas}
\address{Department of Mathematics, University of Maryland,
4176 Campus Dr,
College Park, MD 20742, USA}
\email{tdarvas@umd.edu}
\date{}

\author{Yaxiong Liu}
\address{Mathematical Science Research Center, Chongqing University of Technology, No. 69, Hongguang Avenue, Banan District, Chongqing 400054, China.
\newline\hspace*{\parindent}%
Department of Mathematics, University of Maryland,
4176 Campus Dr,
College Park, MD 20742, USA}
\email{yxliu269@cqut.edu.cn, yxliu238@umd.edu}
\date{}

\begin{abstract}
We establish the long time existence of the inverse Monge--Amp\`ere flow on a klt $\mathbb{Q}$-Fano variety $X$. We further prove that $X$ admits a K\"ahler--Einstein metric if and only if the flow converges in $L^1$.
\end{abstract}


\maketitle

\tableofcontents

\section{Introduction}

Let $X$ be a klt $\mathbb{Q}$-Fano variety.  That is, $X$ is a normal projective variety with klt singularities and ample anticanonical $\mathbb{Q}$-Cartier divisor $-K_X$. In this work, we study K\"ahler--Einstein metrics on such varieties from the perspective of gradient flows.

Fix a sufficiently divisible integer $p>0$ such that $-pK_X$ is very ample, and let $
\iota\colon X\hookrightarrow\mathbb{P}^N$ be a projective embedding induced by sections of $-pK_X$. We fix the smooth K\"ahler form
$$
\omega:=\frac{1}{p}\,\iota^*\omega_{\mathrm{FS}}
$$
on $X$, where $\omega_{\mathrm{FS}}$ denotes the Fubini--Study form of $\mathbb{P}^N$. Thus $\{\omega\} = c_1(X)$. Let $\mathcal H_\omega$ denote the corresponding space of K\"ahler potentials, using the convention $\ddc = \frac{i}{2\pi} \partial \bar \partial$:
$$\mathcal H_\omega := \{u \in C^\infty(X) \textup{ s.t. } \omega_u := \omega + \ddc u > 0 \}.$$
By $\PSH_\omega$ we denote the corresponding space of $\omega$-psh functions. By \cite[Theorem B]{CGZ13} we can simply think of $\PSH_\omega$ as restrictions to $X$ of $\frac{1}{p} \omega_{FS}$-psh functions on $\Bbb C \Bbb P^{N}$, that are not identically equal to $-\infty$.

The search for canonical K\"ahler metrics has its origins in Calabi's conjectures \cite{Calabi54,Calabi57} and Yau's solution of the Calabi conjecture \cite{Yau78a}. In the present setting, where $X$ is a normal klt Fano variety, the natural canonical object is a weak K\"ahler--Einstein (KE) metric satisfying
\begin{equation}\label{eq: KE_eq}
\Ric(\omega_u)=\omega_u
\qquad\text{on }X_{\reg}.
\end{equation}
In accordance with the regularity theory for degenerate complex Monge--Amp\`ere equations on singular spaces \cite{EGZ09,BBGZ13,  BBEGZ19,CC25}, we require
$
u\in\PSH_\omega\cap C^0(X)\cap C^\infty(X_{\reg}).
$ \smallskip

Even for smooth Fano manifolds, the K\"ahler--Einstein equation
\eqref{eq: KE_eq} may fail to admit a solution. Although the variational
framework is discussed in detail in the preliminaries, we recall its
main features here. This approach enlarges the space of smooth K\"ahler
potentials $\mathcal H_\omega$ to the larger space $\mathcal E^1_\omega$
of finite-energy $\omega$-psh potentials. On this space, we define the
functional
\[
    D\colon \mathcal E^1_\omega\longrightarrow \Bbb R,
    \qquad
    D(u):=-E(u)+L(u),
\]
where
\[
    L(u):=-\log\int_X e^{-u}\,d\mu.
\]
Here $E\colon\mathcal E^1_\omega\to\Bbb R$ denotes the
Monge--Amp\`ere energy, and $\mu$ is the adapted measure associated
with $-K_X$. By \cite[Theorem 4.8]{BBEGZ19} the minimizers of $D$ on $\mathcal E^1$ are precisely the solutions to \eqref{eq: KE_eq}. 

The functional $D$ was apparently known to experts as early as the
1970s \cite{Yau78a}. Its first explicit appearance seems to be in
the work of Bando--Mabuchi \cite{BM87}, followed shortly thereafter
by Ding \cite{Ding88}. It was subsequently studied systematically by
Ding--Tian \cite{DT92} and is now often referred to as the Ding
functional.

Rather than seeking minimizers of $D$ directly, we adopt a gradient-flow
approach in this work. Starting from an arbitrary potential \(u_0 \in \mathcal H_\omega\), we deform it along the negative gradient flow of
 \(D\) with respect to the Mabuchi metric on
\(\mathcal H_\omega\). This flow was introduced and studied on smooth
Fano manifolds in \cite{CHT22} and its long time asymptotics was further investigated in \cite{Xia21,His23,Kl24}.

We will give the details of obtaining the flow equation in the preliminaries. For now, we state
the resulting evolution equation in our singular klt setting. Throughout the paper, \(C_x^{(\cdot)}C_t^{(\cdot)}\) denotes a space of functions with \(C_x^{(\cdot)}\)-regularity in the spatial direction and \(C_t^{(\cdot)}\)-regularity in the time direction.

Given
\(u_0\in\cH_\omega\), we seek
\begin{equation}\label{eq: reg_condition}
u\in
{C_x^{\infty}C_t^{1}}\bigl(X_{\mathrm{reg}}\times(0,\infty)\bigr)
\cap
C\bigl(X\times[0,\infty)\bigr)
\cap
{C_x^{0}C_t^{0,1}}\bigl(X\times(0,\infty)\bigr).
\end{equation}
satisfying
\begin{equation}\label{eq:iMAf_X_reg_intr}
\begin{cases}
\frac{1}{V_\omega}(1-\dot{u})(\omega+\ddc u)^n
  =e^{-u+L(u)}\mu
& \text{on } X_{\reg}\times(0,\infty), \\[1mm]
u_t:= u(t,\cdot)\in\PSH_\omega
& \text{on } X\times\{t\}, \\[1mm]
\dot u:=\frac{\partial}{\partial t} u \in L^\infty_{\textup{loc}} & \text{on }X \times (0,\infty) \\[1mm]
u=u_0
& \text{on } X\times\{0\}.
\end{cases}
\end{equation}

Our regularity assumption \eqref{eq: reg_condition} ensures that all derivatives appearing in the above system are well defined in the classical sense. Our main result establishes that this system admits a unique solution for all time. Moreover, the resulting flow converges to a Kähler–Einstein potential if and only if such a potential exists:
\begin{theorem}
\label{thm: main_theorem}
    Let $(X,\omega)$ be a klt Fano K\"ahler variety as described above and let $u_0\in\cH_\omega$. \vspace{1mm}\\
(i) (existence) There exists a unique solution $u$ of the inverse MA flow equation \eqref{eq:iMAf_X_reg_intr} on $X\times[0,\infty)$.\vspace{1mm}\\
(ii) (convergence)  $X$ admits a KE metric if and only if there exists $u_\infty \in \PSH_\omega \cap C^0(X)$  such that $u_t \to u_\infty$  in $L^1(X)$ as $t \to \infty$. In this latter case $\omega_{u_\infty}$ is a KE metric that is smooth on $X_\reg$.
\end{theorem}

We establish existence via a detailed approximation scheme on a log resolution of $(X,-K_X)$. After resolving the singularities, we perturb the pullback of the parabolic system and derive a priori estimates that are uniform in the perturbation parameter (see Section~\ref{sec:Existence_inv_MA_flow}). Uniqueness and convergence are established using the Mayer--Streets theory of negative gradient flows in K\"ahler geometry \cite{May,Streets}, making contact with the $\operatorname{CAT}(0)$ geometry of $\mathcal{E}^2_\omega$ (see Section~\ref{sec:Convergence of inverse MA flow}).

Although we do not pursue the algebraic ramifications of our study, we recall that the solvability of \eqref{eq: KE_eq} can be characterized in terms of the algebro-geometric notion of K-stability. This notion was introduced by Tian \cite{Tian97} and subsequently refined by Donaldson \cite{Don02}. Algebraic characterizations of smooth KE manifolds have been established through a substantial body of work \cite{CDS15,TG15,DS16,CSW18,BBJ21,Zh24}. 
Analogous results for singular klt \(\mathbb{Q}\)-Fano  KE varieties have been obtained in \cite{Berman16,LTW21,LTW22,Li22a,LXZ22}, among many other contributions to this rapidly developing literature. 

For related developments concerning the K\"ahler--Ricci flow on
varieties with log terminal singularities, see
\cite{ST17,BBEGZ19}. A parabolic pluripotential theory for
degenerate complex Monge--Amp\`ere equations, with applications
to such flows, was developed in \cite{GLZ20}.
\smallskip

\noindent \textbf{Acknowledgments.} 
The first named author was partially supported by National Key R\&D Program of China 2024YFA1014800 and 2023YFA1009900, and NSFC grants 12471052 and 12271008. The third named author was partially supported by NSFC grant 12601149.
Part of this work was carried out while the third-named author was visiting the Institute for Theoretical Sciences at Westlake University, which he would like to thank for the hospitality and support.

\smallskip
\noindent \textbf{Use of AI.} This project was initiated in the Fall of 2025. We developed the perturbation scheme presented in Section 4 and identified the apriori estimates required to make the Mayer--Streets minimizing movement gradient flow formalism applicable.   We subsequently established most of these estimates, with the notable exception of Proposition 4.11 that was critically assisted by ChatGPT. Apart from these contributions and AI assisted language editing, the mathematical arguments were written by the authors.

\section{Preliminaries}

\subsection{Klt Fano varieties}

We recall some basic notions and results from the variational theory of
canonical metrics on klt $\mathbb{Q}$-Fano varieties, following
\cite{BBEGZ19,Dar17a,DNG18}. Let $X$ be a klt $\mathbb{Q}$-Fano variety
of dimension $n$, and let $\omega$ be a smooth K\"ahler form on $X$, as
described in the introduction, such that $[\omega]=c_1(X)$.

The klt condition means that, for some (equivalently, for every) log resolution
\begin{equation}\label{eq: pi_log_resol}
\pi\colon Y\longrightarrow X,
\end{equation}
(which is an isomorphism over \(X_{\reg}\)) the exceptional divisor
$
\textup{Exc}(\pi)=\bigcup_j E_j
$
has simple normal crossings and
$$
K_Y=\pi^*K_X+\sum_j a_jE_j
$$
as \(\mathbb{Q}\)-divisors on \(Y\), where \(a_j>-1\) for every \(j\).

\subsection{Pluripotential preliminaries.} 
Following \cite{BEGZ10,BBEGZ19}, for any $u\in\PSH_\omega$, the
associated non-pluripolar Monge--Amp\`ere measure is defined by
\begin{align*}
    \omega_u^n
    \coloneqq
    \lim_{k\to\infty}
    \mathbf{1}_{\{u>-k\}}
    \left(\omega+\ddc\max\{u,-k\}\right)^n.
\end{align*}
This is a positive Borel measure on $X$ satisfying
\[
    \int_X\omega_u^n\leq V_\omega
    \coloneqq\int_X\omega^n.
\]
For $p\geq1$, the full-mass class and the finite-energy class are
defined by
\begin{align*}
    \cE_\omega\coloneqq
    \left\{
        u\in\PSH_\omega
        \ \middle|\
        \int_X\omega_u^n=V_\omega
    \right\}, \ \ \  \ \ 
    \cE^p_\omega\coloneqq
    \left\{
        u\in\cE_\omega
        \ \middle|\
        \int_X|u|^p\omega_u^n<\infty
    \right\}.
\end{align*}
Thus we have the strict inclusions 
\[
    \cH_\omega \subset\cE^p_\omega \subset\cE_\omega \subset\PSH_\omega.
\]
It follows from \cite{DNG18} that, for every $p\geq1$, the $L^p$ type Mabuchi metric $d_p$ on $\cE^p_\omega$ is complete, moreover 
$\cH_\omega \subset \cE^p_\omega$ is $d_p$-dense. Additionally, $\bigl(\cE^p_\omega,d_p\bigr)$ is a
geodesic metric space, with a distinguished class of geodesic segments coming from solutions to a degenerate complex Monge--Amp\`ere equation.  The case $p=1$ was obtained independently in
\cite{Dar17a}. In particular, \cite[Theorem~4.9]{DNG18} shows that
$\bigl(\cE^2_\omega,d_2\bigr)$ is a $\mathrm{CAT}(0)$ space, shadowing the analogous result from the K\"ahler case \cite{Dar17, CC02, Mab87}.

For $u\in\cH_\omega$, the Monge--Amp\`ere energy is defined by
\[
    E(u) \coloneqq \frac{1}{(n+1)V_\omega}\sum_{j=0}^n \int_X u\,\omega_u^j\wedge\omega^{n-j}.
\]
The same formula extends to bounded $\omega$-psh potentials. Using the
monotonicity of $E$, one then defines, for any $v\in\PSH_\omega$,
\[
    E(v) \coloneqq \inf\left\{ E(u) \ \middle|\ u\in\PSH_\omega\cap L^\infty(X),\ u\geq v \right\} \in[-\infty,\infty).
\]
We have $E(v)>-\infty$ if and only if
$v\in\cE^1_\omega$. Moreover, $E$ is affine along finite-energy weak
geodesics, reflecting the close relation of the Monge--Amp\`ere energy to the $d_1$-geometry of $\cE^1_\omega$. \medskip

\subsection{The adapted measure and KE metrics.} Now we discuss the construction of the adapted measure $\mu$ of $-K_X$. 
Let $r$ be a positive integer such that $rK_X$ is Cartier. Let $h$ be a smooth Hermitian metric on the line bundle $rK_X$ such that. For a nowhere vanishing holomorphic section $\sigma$ of $rK_X$ over an open set $U$,  $|\sigma|_{h}$ is a smooth function on $U$.
The adapted measure is locally defined by
\begin{align*}
    \mu\coloneqq \frac{\big(i^{rn}\sigma\wedge\bar{\sigma}\big)^{\frac{1}{r}}}{|\sigma|_{h}^{\frac{2}{r}}}.
\end{align*}
Thus defined $\mu$ depends on $h$, and we choose $h$ so that $\textup{Ric} \mu = \omega$. 
The  klt condition of $X$ yields that the adapted measure $\mu$ has $L^{p_0}$-density relative to $\omega^n$ for some $p_0>1$ (see \cite[Lemma~3.2]{BBEGZ19}).

Following \cite[\S3.2]{BBEGZ19}, we recall the definition of a KE metric on $X$. For $u\in\cE^1_\omega$, the $(1,1)$-current $
\omega_u\coloneqq\omega+\ddc u
$
is called a \emph{KE metric} if
\begin{align}
\frac{1}{V_\omega}\omega_u^n
=\mu_{\omega_u}
\coloneqq e^{-u+L(u)}\mu,
\end{align}
where
\begin{align*}
L(u)\coloneqq-\log\int_X e^{-u}\mu.
\end{align*}
Although defined a priori as a current, any such $\omega_u$ has a continuous potential $u\in C^0(X)$ and is smooth on $X_{\reg}$; see \cite[Lemma~3.6]{BBEGZ19}.

Most importantly, the variational characterization \cite[Theorem~4.8]{BBEGZ19} states that $\omega_u$ is KE if and only if $u$ minimizes the Ding functional
$$
D(u)\coloneqq -E(u) +L(u), \ \ \ u \in \cE^1_\omega.
$$

\subsection{The inverse MA flow.}
We give the definition of inverse Monge--Amp\`ere (MA) flow on klt Fano varity $X$ with $\omega\in c_{1}(X)$. When $(X,\omega)$ is a smooth Fano manifold, in \cite{CHT22}, Collins--Hisamoto--Takahashi introduced the inverse Monge--Amp\`ere (MA) flow, which arises as the gradient flow of Ding functional with respect to the Mabuchi $L^2$ metric. A basic calculation gives the following flow equation:
\begin{align}\label{eq: CHT_flow_eq}
	\dot u=1-e^{\rho_u},
\end{align}
where $\rho_u$ denotes the normalized Ricci potential of the K\"ahler metric $\omega_u$, i.e
\[
\Ric(\omega_{u}) = \omega_{u}+\ddc\rho_{u}, \quad
\int_{X}e^{\rho_{u}}\omega_{u}^{n} = V_{\omega}.
\]
Direct calculation shows
\[
\rho_{u} = -u-\log\left(\frac{\omega_{u}^{n}}{\omega^{n}}\right)
+\rho_{0}-\log\left(\frac{1}{V_{\omega}}\int_{X}e^{-u+\rho_{0}}\omega^{n}\right).
\]
Note that for a smooth Fano manifold, the adapted measure $\mu$ is simply $e^{\rho_{0}}\omega^{n}$. Thus \eqref{eq: CHT_flow_eq} can be rewritten as
\begin{align}
\label{eq:iMAf-smooth}
\frac{1}{V_\omega}(1-\dot u)\omega_u^n
= e^{-u+L(u)+\rho_{0}}\omega^{n}
= e^{-u+L(u)}\mu.
\end{align}
The advantage of this formulation is that both sides of the equation
are well defined on klt $\mathbb{Q}$-Fano varieties, making it the
natural starting point for our analysis.

\begin{definition}\label{definition of IMAF}
Let $(X,\omega)$ be a klt normal Fano K\"ahler variety. We say that a function $u(x,t)$ on $X\times[0,\infty)$ is a solution of the inverse MA flow starting from $u_0\in\cH_\omega$ if 
$$u\in
{C_x^{\infty}C_t^{1}}\bigl(X_{\mathrm{reg}}\times(0,\infty)\bigr)
\cap
C\bigl(X\times[0,\infty)\bigr)
\cap
{C_x^{0}C_t^{0,1}}\bigl(X\times(0,\infty)\bigr)$$
satisfies 
\begin{equation}\label{eq:iMAf-X_reg}
\begin{cases}
\frac{1}{V_\omega}(1-\dot{u})(\omega+\ddc u)^n
  =e^{-u+L(u)}\mu
& \text{on } X_{\reg}\times(0,\infty), \\[1mm]
u_t:= u(t,\cdot)\in\PSH_\omega
& \text{on } X\times\{t\}, \\[1mm]
\dot u:=\frac{\partial}{\partial t} u \in L^\infty_{\textup{loc}} & \text{on }X \times (0,\infty) \\[1mm]
u=u_0
& \text{on } X\times\{0\}.
\end{cases}
\end{equation}
\end{definition}
\bigskip

Fix a smooth resolution $\pi:Y\rightarrow X$ with exceptional divisor $E \subset Y$ as in \eqref{eq: pi_log_resol}. By \cite[Lemma 3.2]{BBEGZ19}, one has
 \begin{align*}
 	\hat\mu\coloneqq\pi^*\mu=e^{\psi^+-\psi^-} \, \d V,
 \end{align*}
where $\d V$ is a smooth volume form on $Y$, and $\psi^{\pm}$ are quasi-plurisubharmonic functions with divisorial analytic singularities, smooth outside $E$, with the singularity types of $\psi^+/\psi^-$ determined by positive/negative parts of $E = E^+ + E^-$. 
By the klt condition, $e^{-\psi^-}\in L^{p_{0}}$ for some 
\begin{equation}\label{eq: p_0_def}
p_{0}>1.
\end{equation}
that we fix for the remainder of the paper. 

Slightly abusing precision, we will often not distinguish $X_\reg$ and $Y\backslash E$, and we will often identify the spaces $\mathcal{H}_\omega(X)/\mathcal{E}^1_\omega(X)/\PSH_\omega(X)$ with $\mathcal{H}_{\pi^* \omega}(Y)/\mathcal{E}^1_{\pi^* \omega}(Y)/\PSH_{\pi^* \omega}(Y)$. \smallskip

After pulling it back to $Y$, the inverse MA flow \eqref{eq:iMAf-X_reg} is equivalent to
\begin{equation}\label{eq:pullback_IMAF}
\begin{cases}
\frac{1}{V_\omega}(1-\dot{u})(\pi^* \omega+\ddc u)^n
 	=e^{-u+L(u)+\psi^+-\psi^-}\,\d V
& \mbox{in $Y\backslash E\times(0,\infty)$},  \\[1mm]
u_t\in\PSH_{\pi^* \omega}(Y) & \mbox{on $Y\times\{t\}$}, \\[1mm]
u  = u_0 & \mbox{on $Y\times\{0\}$},
\end{cases}
\end{equation}
where 
\begin{align*}L(u)= -\log\int_Y e^{-u+\psi^+-\psi^-}\, \d V.
\end{align*}

\section{Inverse MA type flow on smooth manifolds}
\label{IMAF on general manifolds}

In this section, we generalize the long time existence result of \cite{CHT22} to general K\"ahler manifolds. This is crucial when we construct the perturbed flow and some auxiliary flows in Section \ref{sec:Existence_inv_MA_flow}.

Let $(Y,\alpha)$ be a compact K\"ahler manifold, let $dV$ be a smooth
volume form, and let
\[
    f\in C^\infty(Y\times[0,\infty))
\]
be a positive function. We define
\[
    L_t(v)\coloneqq-\log\int_Y e^{-v(x)}f(x,t)\,dV(x).
\]
For $\lambda \in [0,1]$ and any  $u_0\in\cH(Y,\alpha)$, we consider the following inverse MA type flow starting from $u_{0}$:
\begin{align}
\label{eq:alpha_IMAF}
\begin{cases}
V_{\alpha}^{-1}(1-\dot{u})(\alpha+\ddc{u})^n
 	=e^{-\lambda u+L_t(\lambda u)} f \d V & \mbox{in $Y\times[0,\infty)$},  \\[1mm]
u_t \in \cH_\alpha & \mbox{on $Y\times\{t\}$}, \\[1mm]
u(0)  = u_0 & \mbox{on $Y\times\{0\}$},
\end{cases}
\end{align}
where  $ V_{\alpha} \coloneqq \int_{Y}\alpha^{n}$.
Following the arguments of \cite[Proposition 2.6, 4.1, 4.3]{CHT22} with some refinements, we prove the long time existence for solutions to the above flow equation.

\begin{theorem}
\label{thm:alpha_IMAF_exist}
For any $u_0 \in \mathcal H_\alpha$ there exists a smooth solution of \eqref{eq:alpha_IMAF} on $Y \times [0,\infty)$.
\end{theorem}

For fixed initial data, uniqueness also holds and follows from the (standard) parabolic maximum principle applied to the difference of two smooth solutions. Since  this will not be needed later we leave it to the interested reader to work out the details. 

We note that for later applications only the cases $\lambda =0,1$ will be used, but as elaborted below, the argument can be given for general $\lambda \geq 0$ the same way.

\begin{proof}
Standard parabolic theory first gives a smooth solution on a short time interval. By Propositions \ref{prop:alpha_IMAF_C0_est} and \ref{alpha_IMAF_C2_est} below, both $\alpha_u$ and $1-\dot u$ remain uniformly bounded above and below on every finite time interval; hence, after writing the equation as
\[
\dot u=1-V_\alpha e^{-\lambda u+L_t(\lambda u)}f\frac{\d V}{\alpha_u^n},
\]
the resulting operator is uniformly parabolic and concave in the complex Hessian  $i\partial \bar \partial u$, while its dependence on $u$ and $L_\alpha(u)$ is  uniformly controlled. The parabolic complex Evans--Krylov estimate \cite{Kry82,Chu16} and parabolic Schauder theory \cite{Lie96} therefore give uniform higher-order estimates, as in the direct treatment of the inverse Monge--Amp\`ere flow in \cite[Proposition~2.7 and Corollary~2.8]{CHT22}. The standard continuation argument yields long time existence.
\end{proof}

\begin{proposition}[$C^{0}$ estimate]
\label{prop:alpha_IMAF_C0_est}
Suppose $u$ is a smooth solution of \eqref{eq:alpha_IMAF} on $Y\times[0,T)$. Then there is a constant $C=C(T,Y,\alpha,f,\d V,u_{0},\lambda)$ such that
\[
\|u\|_{L^{\infty}(Y\times[0,T))}\leq C.
\]
\end{proposition} 
\begin{proof}Consistent with \cite{CHT22}, we introduce the notation $$e^\rho\coloneqq 1-\dot u.$$
Taking the logarithm of \eqref{eq:alpha_IMAF}, we obtain
\begin{equation}\label{eq:log alpha_IMAF}
\log V_\alpha^{-1}+\rho+\log\alpha_u^n
=-\lambda u+L_t(\lambda u)+\log f+\log \d V.
\end{equation}
Differentiating this identity and using $\dot u=1-e^\rho$ gives
\begin{equation}\label{alpha_IMAF_evol_rho}
\dot\rho
=\Delta_{\alpha_u}(e^\rho)+\lambda(e^\rho-1)
+\frac{d}{dt}L_t(\lambda u)+\partial_t\log f.
\end{equation}
Let $\rho_{\min}(t)\coloneqq\min_Y\rho(\cdot,t)$ and set $M_T\coloneqq
\|\partial_t\log f\|_{L^\infty(Y\times[0,T])}.$
Hamilton's trick and \eqref{alpha_IMAF_evol_rho} show that at almost every
$t$, 
\[
\frac{\d}{\d t}\rho_{\min}(t) 
\geq -\lambda+\frac{d}{dt}L_t(\lambda u)-M_T.
\]
Consequently,
\begin{equation}\label{alpha_IMAF_C0_eqn_1}
\rho-L_t(\lambda u)
\geq \inf_Y\rho(0)-L_0(\lambda u)-(\lambda+M_T)T
\qquad\text{on }Y\times[0,T).
\end{equation}
The Monge--Amp\`ere energy is constant along the flow. Indeed,
\begin{align*}
\frac{\d}{\d t}E(u_t)
&=\frac{1}{V_\alpha}\int_Y\dot u\,\alpha_u^n =1-\int_Y e^{-\lambda u+L_t(\lambda u)}f\,\d V
=0,
\end{align*}
where the last equality follows from the definition of $L_t(\lambda u)$. Since
$\dot u<1$ and $E(u_t)\leq\sup_Yu_t$, it follows that
\begin{equation}\label{alpha_IMAF_C0_eqn_2}
E(u_0)\leq\sup_Yu(\cdot,t)
\leq\|u_0\|_{L^\infty(Y)}+T.
\end{equation}

Combining this with the Zeriahi's uniform Skoda integrability theorem \cite[Corollary 3.2]{Zer01} and the weak $L^1$-compactness result of \cite{GZ07} (c.f. \cite[Lemma 4.13]{Dar19}) we obtain
\begin{equation}\label{alpha-IMAF C0 eqn 3}
	\int_{Y}e^{-pu}\alpha^{n} = e^{-p \sup_X u}\int_{Y}e^{-p(u - \sup_Y u)}\alpha^{n}\leq e^{-p \sup_Y u} C_{p,\alpha}  \leq e^{-p E(u_0)} C_{p,\alpha}  \leq C_{p,u_0, \alpha} \quad \text{, $p>1$}.
\end{equation}
Now rewrite the equation \eqref{eq:alpha_IMAF} as
\begin{equation}\label{alpha-IMAF C0 eqn 4}
\alpha_u^n=V_\alpha e^{-\lambda u+L_t(\lambda u)-\rho}f\,\d V.
\end{equation}
Combining \eqref{alpha_IMAF_C0_eqn_1}, \eqref{alpha-IMAF C0 eqn 3} and  Kolodziej's $L^\infty$-estimate \cite{Kol98},
\[
\sup_{Y}u(\cdot,t)-\inf_{Y}u(\cdot,t) \leq C.
\]
The $C^{0}$-estimate follows from this and \eqref{alpha_IMAF_C0_eqn_2}.
\end{proof}

\begin{proposition}[$C^{2}$ estimate]
\label{alpha_IMAF_C2_est}
Suppose $u$ is a smooth solution of \eqref{eq:alpha_IMAF} on $Y\times[0,T)$. Then there is a constant $C=C(T,Y,\lambda, \alpha,f,dV,u_{0})$ such that
\[
C^{-1}\alpha\leq\alpha_{u}\leq C\alpha.
\]
\end{proposition}

\begin{proof}
Fix $T'\in(0,T)$ and set
\[
\cL\coloneqq\Delta_{\alpha_u}-e^{-\rho}\partial_t.
\]
All constants below are independent of $T'$. We first prove the upper bound
$\alpha_u\leq C\alpha$. By Yau's Schwarz lemma \cite{Yau78b} (see also
\cite[Chapter~3, Lemma~3.7]{Sze14}),
\begin{equation}\label{alpha_IMAF_C2_Schwarz}
\Delta_{\alpha_u}\log\tr_\alpha\alpha_u
\geq-B\tr_{\alpha_u}\alpha-
\frac{\tr_\alpha\Ric(\alpha_u)}{\tr_\alpha\alpha_u},
\end{equation}
where $-B$ is a lower bound for the bisectional curvature of $\alpha$.
Since $L_t(\lambda u)$ is spatially constant, \eqref{eq:log alpha_IMAF}
gives
\begin{align*}
\tr_\alpha\Ric(\alpha_u)
={}&\Delta_\alpha\log(1-\dot u)
+\lambda\Delta_\alpha u-\Delta_\alpha\log f
+\tr_\alpha\Ric(\d V)\\
={}&-\frac{\Delta_\alpha\dot u}{1-\dot u}
-\frac{|\nabla\dot u|_\alpha^2}{(1-\dot u)^2}
+\lambda(\tr_\alpha\alpha_u-n)-\Delta_\alpha\log f
+\tr_\alpha\Ric(\d V)\\
\leq{}&-e^{-\rho}\partial_t\tr_\alpha\alpha_u
+\lambda\tr_\alpha\alpha_u+C_T.
\end{align*}
Substituting this into \eqref{alpha_IMAF_C2_Schwarz}, we obtain, whenever
$\tr_\alpha\alpha_u\geq1$,
\begin{equation}\label{alpha_IMAF C2 log trace}
\cL\log\tr_\alpha\alpha_u
\geq-B\tr_{\alpha_u}\alpha-C_T.
\end{equation}
On the other hand,
\begin{equation}\label{alpha_IMAF_evol_u}
\cL u
=n+1-\tr_{\alpha_u}\alpha-e^{-\rho}
\leq n+1-\tr_{\alpha_u}\alpha.
\end{equation}

Choose $A=B+1$ and let
$Q_1\coloneqq\log\tr_\alpha\alpha_u-Au$. Suppose that $Q_1$
attains its maximum on $Y\times[0,T']$ at $(x_1,t_1)$. If $t_1=0$, or if
$\tr_\alpha\alpha_u(x_1,t_1)<1$, the maximum of $Q_1$ is bounded using
the initial data and the $C^0$-estimate, and hence
$\tr_\alpha\alpha_u\leq C_T$. Otherwise, the maximum principle and
\eqref{alpha_IMAF C2 log trace}--\eqref{alpha_IMAF_evol_u} give
\[
0\geq\cL Q_1
\geq(A-B)\tr_{\alpha_u}\alpha-C_T-A(n+1)
\]
at $(x_1,t_1)$. Hence
\[
\tr_{\alpha_u}\alpha(x_1,t_1)\leq C_T.
\]
Moreover, Proposition~\ref{prop:alpha_IMAF_C0_est},
\eqref{alpha_IMAF_C0_eqn_1}, and \eqref{eq:alpha_IMAF} imply
$$
\alpha_u^n\leq C_T\alpha^n
\qquad\text{on }Y\times[0,T).
$$
The algebraic trace inequality therefore yields, at $(x_1,t_1)$,
\[
\tr_\alpha\alpha_u\leq\frac{1}{(n-1)!}
(\tr_{\alpha_u}\alpha)^{n-1}\frac{\alpha_u^n}{\alpha^n}
\leq C_T.
\]
Since $u$ is uniformly bounded, this bounds the maximum of $Q_1$, and hence
$\tr_\alpha\alpha_u\leq C_T$ throughout $Y\times[0,T']$. Thus
\begin{align}
    \label{alpha_IMAF_C2_metr_upper}
    \alpha_u\leq C_T\alpha.
\end{align}

We next prove the lower bound. The $C^0$-estimate and the smooth positivity
of $f$ imply that $L_t(\lambda u)$ is uniformly bounded on $[0,T)$. It
follows from \eqref{eq:alpha_IMAF} that
\begin{equation}\label{alpha_IMAF C2 eqn 2}
C_T^{-1}e^\rho\leq\frac{\alpha^n}{\alpha_u^n}
\leq C_Te^\rho.
\end{equation}
Combining this with \eqref{alpha_IMAF_C2_metr_upper}, we obtain
\begin{equation}\label{alpha_IMAF_C2_inv_trace}
\tr_{\alpha_u}\alpha
\leq\frac{1}{(n-1)!}(\tr_\alpha\alpha_u)^{n-1}
\frac{\alpha^n}{\alpha_u^n}
\leq C_Te^\rho.
\end{equation}
It remains to bound $\rho$ from above.

Put
\[
Q_2\coloneqq
\rho-\log\bigl(u+\|u\|_{L^\infty(Y\times[0,T))}+1\bigr).
\]
Notice that
\[
1\leq u+\|u\|_{L^\infty(Y\times[0,T))}+1
\leq2\|u\|_{L^\infty(Y\times[0,T))}+1.
\]
To account for the time dependence of $f$, let
\[
\d\nu_t\coloneqq e^{-\lambda u+L_t(\lambda u)}f\,\d V.
\]
This is a probability measure, and direct differentiation gives
\begin{equation}\label{alpha_IMAF C2_deriv_L}
\frac{\d}{\d t}L_t(\lambda u)
=\int_Y\bigl(\lambda\dot u-\partial_t\log f\bigr)\,\d\nu_t.
\end{equation}
If $M_T\coloneqq\|\partial_t\log f\|_{L^\infty(Y\times[0,T])}$, then
$\dot u<1$ and \eqref{alpha_IMAF C2_deriv_L} imply
$$
\lambda-\frac{\d}{\d t}L_t(\lambda u)-\partial_t\log f
\geq-2M_T.
$$
Using \eqref{alpha_IMAF_evol_rho}, we consequently obtain
\begin{align}
\cL\rho
&=-|\nabla\rho|_{\alpha_u}^2-\lambda
+e^{-\rho}\left(
\lambda-\frac{\d}{\d t}L_t(\lambda u)-\partial_t\log f
\right) \notag\\
&\geq-|\nabla\rho|_{\alpha_u}^2-\lambda-2M_Te^{-\rho}.
\label{alpha_IMAF C2_evol_rho L}
\end{align}

Let $(x_2,t_2)$ be a maximum point of $Q_2$ on $Y\times[0,T']$. If
$t_2=0$, the desired upper bound for $Q_2$ follows from the initial data. If
$\rho(x_2,t_2)\leq0$, the preceding lower bound gives
$Q_2(x_2,t_2)\leq0$. We may therefore assume that $t_2>0$ and
$\rho(x_2,t_2)>0$. At this point, \eqref{alpha_IMAF C2_evol_rho L}
gives
\[
\cL\rho\geq-|\nabla\rho|_{\alpha_u}^2-C_T.
\]
Since $\nabla Q_2=0$ at $(x_2,t_2)$,
\[
\nabla\rho
=\frac{\nabla u}
{u+\|u\|_{L^\infty(Y\times[0,T))}+1}.
\]
The maximum principle and \eqref{alpha_IMAF_evol_u} now yield
\begin{align*}
0\geq\cL Q_2
&=\cL\rho
-\frac{\cL u}{u+\|u\|_{L^\infty(Y\times[0,T))}+1}
+\frac{|\nabla u|_{\alpha_u}^2}
{\bigl(u+\|u\|_{L^\infty(Y\times[0,T))}+1\bigr)^2}\\
&\geq-C_T
+\frac{\tr_{\alpha_u}\alpha-n-1}
{u+\|u\|_{L^\infty(Y\times[0,T))}+1}.
\end{align*}
It follows that $\tr_{\alpha_u}\alpha(x_2,t_2)\leq C_T$. On the other
hand, the arithmetic--geometric mean inequality and
\eqref{alpha_IMAF C2 eqn 2} give
\[
\tr_{\alpha_u}\alpha
\geq n\left(\frac{\alpha^n}{\alpha_u^n}\right)^{1/n}
\geq C_T^{-1}e^{\rho/n}.
\]
Thus $\rho(x_2,t_2)\leq C_T$, and hence $Q_2\leq C_T$ on
$Y\times[0,T']$. Since
$u+\|u\|_{L^\infty(Y\times[0,T))}+1$ is uniformly bounded, we conclude
that $\rho\leq C_T$. Inequality \eqref{alpha_IMAF_C2_inv_trace}
now gives $\tr_{\alpha_u}\alpha\leq C_T$, i.e.
$\alpha_u\geq C_T^{-1}\alpha$, completing the  proof.
\end{proof}

\section{Long time existence of the klt inverse MA flow}
\label{sec:Existence_inv_MA_flow}
In this section, we will prove the long time existence part of Theorem \ref{thm: main_theorem}. We fix a log resolution $\pi: Y \to X$, as described in Section 2.4. Recall that 
 \begin{align*}
 	\hat\mu\coloneqq\pi^*\mu=e^{\psi^+-\psi^-} \, \d V,
 \end{align*}
where $\psi^+,\psi^-$ are qpsh potentials with analytic singularity type and that $e^{-\psi^-}\in L^{p_{0}}$ for some 
\begin{equation}\label{eq: p_0_def_Sec_4}
p_{0}>1,
\end{equation}
that we also fix for the remainder of this section. To simplify our discussion, for this section we will also renormalize the flow of \eqref{eq:pullback_IMAF}: 
\begin{equation}\label{eq: normaliz}
u_t := u_t-t.
\end{equation}
Thus, our main flow system \eqref{eq:pullback_IMAF} will become:
\begin{equation}\label{eq:IMAF-normalized}
\begin{cases}
\frac{1}{V_\omega}(-\dot{u})(\pi^* \omega+\ddc u)^n
 	=e^{-u+L(u) + \psi^+ - \psi^-} \,\d V
& \mbox{in $Y\backslash E\times(0,\infty)$},  \\[1mm]
u_t\in\PSH_{\pi^* \omega} & \mbox{on $Y\times\{t\}$}, \\[1mm]
u  = u_0 & \mbox{on $Y\times\{0\}$}.
\end{cases}
\end{equation}

Our main result of this section will be the following one:
\begin{theorem}
\label{thm:longtime_exist_pullback}
There exists a solution $u\in C_x^{\infty}C_t^1((Y\backslash E)\times(0,\infty))\cap C(Y\times[0,\infty))$ of \eqref{eq:IMAF-normalized} satisfying the following properties: 
\begin{enumerate}
\item[(i)] For any $0<T<\infty$, there are constants $\beta\in(0,1)$ (independent of $T$) and $K_{T}>0$ such that
    \begin{align*}
	|u(x,t)-u(x,t')| \leq K_T|t-t'|^{\beta}
    \quad \text{for $x\in Y$ and $t,t'\in[0,T]$}.
    \end{align*}
\item[(ii)] For any $0<\tau<T<\infty$, there is a constant $K_{\tau,T}>0$ such that
$$
|u(x,t)-u(x,t')| \leq K_{\tau,T}|t-t'|
   \quad \text{for $x\in Y$ and $t,t'\in[\tau,T]$}.
$$
\end{enumerate}
\end{theorem}

\subsection{The perturbed flow}
The strategy of proving Theorem \ref{thm:longtime_exist_pullback}
is to  perturb the semiample class $[\pi^* \omega]$ to a K\"ahler class, and at the same time perturb the singular weights $\psi^\pm$  to become smooth. Then we will obtain apriori estimates for the resulting perturbed system that are preserved in the limit.  \smallskip

Let $E_0$ be an effective exceptional divisor such that $A\coloneqq -\pi^*K_X-\gamma E_0$ is $\Bbb Q$-ample for some small $\gamma>0$ rational. 
We fix a K\"ahler metric $\omega_Y\in c_1(A)$ and a Hermitian metric $h_{0}$ on $E_{0}$. Then there exists a smooth function $\psi'$ such that
\[
\omega_{Y} = \pi^* \omega-\gamma R_{h_{0}}+\ddc\psi'.
\]
By Poincar\'e-Lelong formula $\ddc\log|s_{E_0}|^2_{h_0}=[E_0]-R_{h_{0}}$, we see that
\begin{align*}
\pi^* \omega+\ddc(\psi'+\gamma\log|s_{E_0}|_{h_0}^2) = \omega_Y+\gamma[E_0]
\end{align*}
is a K\"ahler current. So 
$$\psi:=\psi'+\gamma\log|s_{E_0}|_{h_0}^2$$ 
is a $\pi^*\omega$-plurisubharmonic function such that
\begin{align}
\label{eq:piomega-omega_Y}
\pi^* \omega+\ddc\psi=\omega_Y+\gamma[E_0].
\end{align}
Now we construct the perturbed flow. 

By Demailly's regularization, there exist a constant \(B_0>0\) and
sequences
\[
\psi_k^\pm\in\operatorname{PSH}(Y,B_0\omega_Y)\cap C^\infty(Y)
\]
such that \(\psi_k^\pm\searrow\psi^\pm\), with convergence in the
\(C^\infty\)-topology on relatively compact subsets of
\(\operatorname{Amp}([\pi^*\omega])\). Set
\[
\varepsilon_k\coloneqq \frac1k,
\qquad
g_{\varepsilon_k}\coloneqq e^{\psi_k^+-\psi_k^-}.
\]
When it will not cause ambiguity, we will drop the subscript $k$ from $\varepsilon_k$. 

Then \(g_{\varepsilon_k}\to g\) in \(L^{p_0}(Y,\mathrm dV)\). Note that $|\sup_Y\psi_k^{\pm}|$ is bounded independently of $k$, while we have for all $k$,
 \begin{align*}
 	\|e^{-\psi_k^-}\|_{L^{p_{0}}}
 	\leq \|e^{-\psi^-}\|_{L^{p_0}}.
 \end{align*}
We introduce the perturbed background K\"ahler metric:
$$\omega_\eps\coloneqq\pi^* \omega+\eps\omega_Y.$$ 
Since $u_0\in\cH_{\pi^* \omega}$, then $u_0\in\cH_{\omega_\eps}$. By Theorem \ref{thm:alpha_IMAF_exist}, using the normalization \eqref{eq: normaliz}, there is a smooth solution of the following perturbed normalized inverse MA flow:
 \begin{align}
 \label{eq:eps_IMAF} 
 	\begin{cases}
V_{\omega_\eps}^{-1}(-\dot{u}^\eps)(\omega_\eps+\ddc{u}^\eps)^n
 	=e^{-u^\eps+L_\eps(u^\eps)} g_\eps \, \d V
& \mbox{in $Y\times[0,\infty)$},  \\[1mm]
u^{\eps}_t\in\PSH_{\omega_{\eps}} & \mbox{on $Y\times\{t\}$}, \\[1mm]
u^\eps_0 = u_0 & \mbox{on $Y\times\{0\}$},
\end{cases}
 \end{align}
where 
\begin{align*}
V_{\omega_{\eps}} \coloneqq \int_{Y}\omega_{\eps}^{n}
\quad\text{and}\quad 
L_\eps(u^\eps)\coloneqq -\log\int_Y e^{-u^\eps} g_\eps \, \d V.
\end{align*}

\subsection{Uniform estimates}

For a fixed $T>0$, we consider the restriction of \eqref{eq:eps_IMAF} on $Y\times[0,T]$:
\begin{align}
\label{eq:eps_IMAF T} 
 	\begin{cases}
V_{\omega_\eps}^{-1}(-\dot{u}^\eps)(\omega_\eps+\ddc{u}^\eps)^n
 	=e^{-u^\eps+L_\eps(u^\eps)} g_\eps \,  \d V
& \mbox{in $Y\times[0,T]$},  \\[1mm]
u^{\eps}_t\in\PSH_{\omega_{\eps}} & \mbox{on $Y\times\{t\}$}, \\[1mm]
u^\eps_0 = u_0 & \mbox{on $Y\times\{0\}$},
\end{cases}
 \end{align}

In this subsection, we will establish several estimates of $u^{\eps}$. Note that all these estimates are uniform (i.e. independent of $\eps$). For ease of notation, in Section \ref{sec:Existence_inv_MA_flow}, a constant is said to be uniform if it is independent of $\eps$.  Sometimes, we use subscript to emphasize its dependence (e.g. $K_{\theta}$ depends on $\theta$).

\begin{proposition}[Uniform Skoda integrability]
\label{prop: uniform_skoda_eps}
For any $q>1$, there is a uniform constant $K_{q}$ (independent of $\varepsilon$) such that
\begin{align*}
\int_Y e^{-qu^{\eps}_t} \d V \leq K_q, \ \ t \in [0,T].
\end{align*}
\end{proposition}

\begin{proof}
By $\dot{u}<0$, we have $\sup_{Y}u^{\eps}_t\leq \sup_{Y}u_{0}$. On the other hand, by the flow equation \eqref{eq:eps_IMAF T} 
\[
\frac{\d}{\d t}E_{\eps}(u^{\eps}_t)
= \frac{1}{V_{\omega_{\eps}}}
\int_Y\dot{u}_t^{\eps}(\omega_{\eps}+\ddc u^{\eps}_t)^n
= -\int_Y e^{-u_t+L_{\eps}(u^{\eps}_t)}g_{\eps}\,\d V = -1,
\]
and so $E_{\eps}(u^{\eps}_t)\geq E_{\eps}(u_{0})-T\geq-K$. Proposition \ref{prop: uniform_skoda_eps} then follows from Lemma \ref{lem: uniform_skoda} below.
\end{proof}

\begin{lemma}
\label{lem: uniform_skoda}
For any $B>0$ and $q>1$, there is a uniform constant $K_{B,q}$ (independent of $\varepsilon$) such that for any $v\in\cE^1_{\omega_\eps}$ satisfying $\sup_{Y}v\leq B$ and $E_\eps(v)\geq-B$,
\begin{align*}
\int_Y e^{-qv} \d V \leq K_{B,q}.
\end{align*}
\end{lemma}

\begin{proof}
Without loss of generality, we assume that $\sup_{Y}v\leq 0$ and $E_\eps(v)\geq-B$. To prove Lemma~\ref{lem: uniform_skoda}, we argue by contradiction. Suppose that there are two sequences $\{\eps_{j}\}$ and $\{v_{j}\}$ such that
\[
v_{j}\in\cE^1_{\omega_{\eps_j}}, \quad
\sup_{Y}v_{j} \leq 0, \quad
E_{\eps_{j}}(v_{j}) \geq -B,
\]
and
\[
\int_{Y}e^{-qv_{j}} \d V \to \infty.
\]
By passing to a subsequence, we can assume that
\[
\eps_{j} \to \eps\in[0,1] 
\quad\text{and}\quad
v_{j} \xrightarrow{L^{1}} v \in \textup{PSH}_{\omega_{\eps}}.
\]

We claim: $E_{\eps}(v)\geq-B$. 

Given this claim, it follows from \cite[Theorem~1.1]{BBEGZ19} that $v$ has zero Lelong numbers. By Demailly-Koll\'ar's semi-continuity theorem for log-canonical thresholds \cite[Main Theorem 0.2 (2)]{DK01} (c.f. \cite{PS00}) we have
\[
\lim_{j\to\infty}\int_{Y}e^{-qv_{j}}\d V 
= \int_{Y}e^{-qv} \d V < \infty,
\]
which is a contradiction.

Next we prove the claim. Consider the bounded cutoffs:
\begin{align*}
v_j^{(k)} := \max(v_{\eps_j}, -k)
\quad\text{and}\quad
v^{(k)} := \max(v,-k).
\end{align*}
Since $E_\varepsilon$ is continuous under decreasing limits \cite{BEGZ10}, it suffices to show that $E_{\eps}(v^{(k)})\geq-B$. Fix a small $0 <  a< 1$ such that $\eps_j \leq \eps + a$ for all $j$ sufficiently large. Then $\omega_{\eps_j}\leq\omega_{\eps+a}$ and $v_j^{(k)} \in \textup{PSH}(Y, \omega_{\eps+a})$ for all large $j$.
Note that for $s\in[\eps_j,\eps+a]$,
\begin{align*}
	E_{s}(v_j^{(k)})
	=\frac{1}{(n+1) V_{\omega_s}}\sum_{l =0}^{n}\int_Y v_j^{(k)}(\omega_{s}+r\ddc v_j^{(k)})^l \wedge \omega_s^{n-l} 
\end{align*}
Since $|v^{(k)}_j,v^{(k)}| < k$, by multilineary of non-pluripolar products we obtain that
\begin{align*}
	|E_{\eps_j}(v_j^{(k)}) - E_{\eps+a}(v_j^{(k)}) |
	\leq C_k|\eps_j-(\eps+a)|.
\end{align*}
Thus, we have 
\begin{align*}
	E_{\eps_j}(v_j^{(k)}) \leq E_{\eps+a}(v_j^{(k)}) + C_k|\eps_j - ( \eps+a)|,
\end{align*}
which yields, by taking the $\limsup$ as $j \to \infty$,
\begin{align*}
	\limsup_{j\rightarrow\infty} E_{\eps_j}(v_j^{(k)}) 
	\leq \limsup_{j\rightarrow\infty} E_{\eps+a}(v_j^{(k)}) + C_ka
	\leq E_{\eps+a}(v^{(k)}) + C_ka.
\end{align*}
Letting $a\rightarrow0$ and using $E_{\eps_j}(v_j^{(k)})\geq E_{\eps_j}(v_j)\geq-B$, we obtain $E_{\eps}(v^{(k)}) \geq-B$, as required.
\end{proof}

Following Chen--Cheng's argument \cite[\S 2.2]{CC23} (building on \cite{GPT23}), we next establish the uniform $C^{0}$ estimate and some related stability estimates that we will need later.

\begin{proposition}[$C^{0}$ estimate]
\label{prop:unif_Linfty}
	There is a uniform constant $K=K_T$  (independent of $\varepsilon$) such that
    \begin{align*}
    	\sup_{Y\times[0,T]}|u^\eps|\leq K.
    \end{align*}
\end{proposition}

\begin{proposition}[Stability estimates]
\label{prop:sub_stab_holder}
For any $\beta\in\left(0,\frac{p_0-1}{(2 + n)p_0-1 }\right)$,
there is a uniform constant 
$K_{\beta}$  (independent of $\varepsilon$)  such that
\begin{align}
\label{eq:quasi-decreasing}
	\sup_{Y\times[0,T]}(u^{\eps}-u^{\eps'})
	\leq K_{\beta}\|u^{\eps}-u^{\eps'}\|_{L^1}^\beta
    \quad \text{for $\eps<\eps'$},
\end{align}
and
\begin{align}
\label{eq:time Holder}
	|u^\eps(x,t)-u^\eps(x,t')|\leq K_{\beta}|t-t'|^\beta
    \quad \text{for $x\in Y$ and $t,t'\in[0,T]$}.
\end{align} 
\end{proposition}

To prove the above results, we first need to argue  a series of lemmas.

\begin{lemma}
\label{eq:unif-Lp}
Denote
\[
e^{F_{\eps}} \coloneqq e^{-u^{\eps}+L_{\eps}(u^{\eps})}g_{\eps}
= e^{-u^{\eps}+L_{\eps}(u^{\eps})+\psi_{\eps}^{+}-\psi_{\eps}^{-}}.
\]
For any $r>n+1$, there is a uniform constant $K_{r}$  (independent of $\varepsilon$) such that
 	\begin{align}
 	\label{eq:unif-r-Ent}
 		\Ent_r(F_\eps)\coloneqq
 		\int_Y(F_\eps^2+1)^{\frac{r}{2}}e^{F_\eps}\,\d V
 		\leq K_{r}.
 	\end{align}
 \end{lemma}
 \begin{proof}
It follows from $\dot{u}^{\eps}<0$ that $u^{\eps}\leq u_{0}$ and then 
 \begin{align*}
 		L_\eps(u^\eps)\leq L_\eps(u_0)
 		=-\log\int_Y e^{-u_0} g_\eps\,\d V
 		\leq K.
 \end{align*}
Then using $\sup_{Y}\psi_{\eps}^{+}\leq K$, H\"older's inequality and Lemma~\ref{lem: uniform_skoda},
\[
\int_Y e^{pF_\eps}\,\d V
\leq e^{K}\int_Y e^{-pu^\eps}e^{-p\psi_\eps^-} \,\d V
\leq e^{K}\left( \int_Y e^{-p(p_0p^{-1})^*u^\eps}\,\d V\right)^{\frac{1}{(p_0p^{-1})^*}} \left( \int_Y e^{-p_0\psi_\eps^-} \,\d V\right)^{\frac{p}{p_0}} \leq K.
\]
Note that for all $x\in\RR$, $(x^2+1)^{r/2} e^x\leq Ce^{px}+C$, and then \eqref{eq:unif-r-Ent} follows.
 \end{proof}

Let $s\geq 0$, and let $\eta_j:\RR\rightarrow\RR_+$ be a family of smooth positive functions $\eta_j$, such that $\max(x,0)< \eta_j(x)< \max(x,0) + \frac{1}{j}$. We consider the following auxiliary flow:
\begin{equation}\label{eq:auxiliary_equ}
\begin{cases}
V_{\omega_\eps}^{-1}(-\dot{\psi}_j)(\omega_\eps+\ddc{\psi}_j)^n
 	=A_{j,s}^{-1}\cdot\eta_j(-u^\eps-s) e^{F_\eps}\,\d V
& \mbox{in $Y\times[0,T]$},  \\[1mm]
\psi_j(t)\in\PSH_{\omega_{\eps}} & \mbox{on $Y\times\{t\}$}, \\[1mm]
\psi_j  = 0 & \mbox{on $Y\times\{0\}$},
\end{cases}
\end{equation}
where 
\begin{align*}
    \psi_j:=\psi_{\eps,j},
    \quad
	A_{j,s}\coloneqq A_{\eps,j,s}\coloneqq \int_{Y\times [0,T]}\eta_j(-u^\eps-s) e^{F_\eps}\d V \d t.
\end{align*}
After applying a smooth reparametrization in time, Theorem \ref{thm:alpha_IMAF_exist} is applicable, with normalization \eqref{eq: normaliz} and $\lambda =0$. Hence there exists a smooth solution of \eqref{eq:auxiliary_equ}.

Note that $\dot\psi_j<0$, hence $\psi_j\leq0$. On the other hand,
\begin{align*}
E_\eps(\psi_j)= {} &E_\eps(\psi_j(0))+\int_0^t\frac{\d}{\d s}E_\eps(\psi_j(s)) \d s \\
= {} &\int_0^{t}V_\eps^{-1}\int_Y\dot\psi_j(\omega_\eps+\ddc\psi_j)^n\, \d t \\
= {} &\int_0^t\int_Y-A_{j,s}^{-1}\eta_j(-u^\eps-s) e^{F_\eps}\, dV \d t
\geq-1.
\end{align*} 
Then for any $q>1$, Lemma~\ref{lem: uniform_skoda} implies
\begin{align}
\label{eq:unif-skoda-psi_j}
	\int_Y e^{-q\psi_j}\d V\leq K_{q}.
\end{align}

\begin{lemma}
\label{lem:Phi_leq0}
	There exist dimensional constants $c_n>0$ and $C_n>0$ such that for any $s\geq\|u_0\|_{L^\infty}$,
	\begin{align*}
		c_n A_{j,s}^{-\frac{1}{n+1}}((-u^\eps-s)_{+})^{\frac{n+2}{n+1}}
		\leq -\psi_j+C_n A_{j,s} \ \ \ \textup{ on } Y \times [0,T].
	\end{align*}
\end{lemma}
\begin{proof}
We introduce the notation
$$\omega'\coloneqq \omega_\eps+\ddc{u}^\eps,$$ 
and define the parabolic operator 
	\begin{align*}
		\cL v\coloneqq \cL_\eps v
		\coloneqq -\p_t v+(-\dot u^\eps)\Delta_{\omega'} v
	\end{align*}
	and the constants
	\begin{align*}
		\lambda\coloneqq\left(\frac{1}{2}n^{\frac{n}{n+1}}\frac{1}{n+2} \right)^{\frac{n+1}{n+2}}A_{j,s}^{-\frac{1}{n+2}}, \quad
		\Lambda\coloneqq \left(\frac{1}{2}n^{\frac{n}{n+1}} \right)^{-(n+1)}\frac{(n+1)^{n+2}}{n+2}A_{j,s}.
	\end{align*}
	We use the test function
	\begin{align*}
		\Phi\coloneqq \lambda(-u^\eps-s)-(-\psi_j+\Lambda)^{\frac{n+1}{n+2}}.
	\end{align*}
It suffices to show $\Phi\leq0$. One computes
	\begin{align*}
		&\cL\Phi
		=-\p_t\Phi+(-\dot u^\eps)\Delta_{\omega'}\Phi \\
		&=-\lambda(-\dot u^\eps)+\frac{n+1}{n+2}(-\psi_j+\Lambda)^{-\frac{1}{n+2}}(-\dot\psi_j)  \\
		&+(-\dot u^\eps)\left( 
		\lambda(-\Delta_{\omega'}u^\eps)
		+\frac{n+1}{n+2}(-\psi_j+\Lambda)^{-\frac{1}{n+2}}\Delta_{\omega'}\psi_j 
		+\frac{n+1}{(n+2)^2}(-\psi_j+\Lambda)^{-\frac{n+3}{n+2}}|\nabla\psi_j|_{\omega'}^2\right) \\
		&\geq-\lambda(-\dot u^\eps)+\frac{n+1}{n+2}(-\psi_j+\Lambda)^{-\frac{1}{n+2}}(-\dot\psi_j)   \\
		&+(-\dot u^\eps)\left( \lambda\tr_{\omega'}\omega_\eps -\lambda n 
		+\frac{n+1}{n+2}(-\psi_j+\Lambda)^{-\frac{1}{n+2}}\tr_{\omega'}(\omega_\eps+\ddc\psi_j)
		-\frac{n+1}{n+2}(-\psi_j+\Lambda)^{-\frac{1}{n+2}}\tr_{\omega'}\omega_\eps	\right).
	\end{align*}
	By the definitions of $\lambda,\Lambda$, one has
	\begin{align*}
		\lambda=\frac{n+1}{n+2}\Lambda^{-\frac{1}{n+2}}
		\geq \frac{n+1}{n+2}(-\psi_j+\Lambda)^{-\frac{1}{n+2}}.
	\end{align*}
	On the other hand, by AM-GM inequality, 
	\begin{align*}
		\tr_{\omega'}(\omega_\eps+\ddc\psi_j)
		\geq n\left(\frac{(\omega_\eps+\ddc\psi_j)^n}{\omega'^n} \right)^{\frac{1}{n}}
		= {} & n\left(\frac{V_\eps(-\dot\psi_j)^{-1} A_{j,s}^{-1}\cdot\eta_j(-u^\eps-s) e^{F_\eps}\, dV}{V_\eps(-\dot{u}^\eps)^{-1} e^{F_\eps}\, dV} \right)^{\frac{1}{n}} \\
		= {} & n\left(\frac{(-\dot u^\eps)\eta_j(-u^\eps-s) }{(-\dot{\psi}_j) A_{j,s}} \right)^{\frac{1}{n}}.
	\end{align*}
	Therefore, we have
	\begin{align*}
		\cL\Phi\geq-\lambda(n+1)(-\dot u^\eps)
		&+\frac{n+1}{n+2}(-\psi_j+\Lambda)^{-\frac{1}{n+2}}(-\dot\psi_j)   \\
		&+n\frac{n+1}{n+2}(-\psi_j+\Lambda)^{-\frac{1}{n+2}}\frac{(-\dot u^\eps)^{1+\frac{1}{n}}}{(-\dot\psi_j)^{\frac{1}{n}}}
		\frac{\eta_j(-u^\eps-s)^{\frac{1}{n}}}{A_{j,s}^{\frac{1}{n}}}.
	\end{align*}
Applying Young’s inequality in the form
\begin{align*}
	\Big(MN^{\frac{1}{n}}\Big)^{\frac{n}{n+1}}x
	\leq Ny+M\frac{x^{1+\frac{1}{n}}}{y^{\frac{1}{n}}}
\end{align*}
with 
\begin{align*}
		y=-\dot\psi_j,\quad x=-\dot u^\eps , \quad
		N=\frac{n+1}{n+2}(-\psi_j+\Lambda)^{-\frac{1}{n+2}}, \quad		M=n\frac{n+1}{n+2}(-\psi_j+\Lambda)^{-\frac{1}{n+2}}\frac{\eta_j(-u^\eps-s)^{\frac{1}{n}}}{A_{j,s}^{\frac{1}{n}}},
	\end{align*}
	yields that
	\begin{align*}
		&\frac{n+1}{n+2}(-\psi_j+\Lambda)^{-\frac{1}{n+2}}(-\dot\psi_j)   
		+n\frac{n+1}{n+2}(-\psi_j+\Lambda)^{-\frac{1}{n+2}}\frac{(-\dot u^\eps)^{1+\frac{1}{n}}}{(-\dot\psi_j)^{\frac{1}{n}}}
		\frac{\eta_j(-u^\eps-s)^{\frac{1}{n}}}{A_{j,s}^{\frac{1}{n}}} \\
		\geq {} & n^{\frac{n}{n+1}}\frac{n+1}{n+2}(-\psi_j+\Lambda)^{-\frac{1}{n+2}}\frac{\eta_j(-u^\eps-s)^{\frac{1}{n+1}}}{A_{j,s}^{\frac{1}{n+1}}}(-\dot u^\eps).
	\end{align*}
	Therefore, we obtain
	\begin{align}
	\label{eq:cL-Phi-geq}
		\cL\Phi
		\geq(-\dot u^\eps)\left(-(n+1)\lambda
		+ n^{\frac{n}{n+1}}\frac{n+1}{n+2}(-\psi_j+\Lambda)^{-\frac{1}{n+2}}\eta_j(-u^\eps-s)^{\frac{1}{n+1}}A_{j,s}^{-\frac{1}{n+1}}
		 \right).
	\end{align}
	Assume that $\Phi$ attains a positive maximum at $(x_0,t_0)$.
	Since $s\geq\|u_0\|_{L^\infty}$, we have $t_0>0$. Since $\Phi(x_0,t_0)>0$, we have 
    \[
    (-\psi_j+\Lambda)^{-\frac{1}{n+2}}(-u^\eps-s)^{\frac{1}{n+1}}
		>\lambda^{-\frac{1}{n+1}},
    \]
    which implies $-u^\eps-s>0$ and so $\eta_j(-u^\eps-s)^{\frac{1}{n+1}}\geq(-u^\eps-s)^{\frac{1}{n+1}}$. It then follows that
	\begin{align*}
		(-\psi_j+\Lambda)^{-\frac{1}{n+2}}\eta_j(-u^\eps-s)^{\frac{1}{n+1}}
		\geq (-\psi_j+\Lambda)^{-\frac{1}{n+2}}(-u^\eps-s)^{\frac{1}{n+1}}
		>\lambda^{-\frac{1}{n+1}}.
	\end{align*}
	By the maximum principle, we have $\cL\Phi(x_0,t_0)\leq0$. Together with \eqref{eq:cL-Phi-geq} and the definition of $\lambda$, we have
	\begin{align*}
		0\geq \cL\Phi(t_0,x_0)
		\geq {} & (-\dot u^\eps)\left(-(n+1)\lambda
		+ n^{\frac{n}{n+1}}\frac{n+1}{n+2}\lambda^{-\frac{1}{n+1}} A_{j,s}^{-\frac{1}{n+1}}
		 \right) \\
		= {} & (-\dot u^\eps)\Big(-(n+1)\lambda + 2(n+1)\lambda\Big) 
		= (-\dot u^\eps)(n+1)\lambda >0,
	\end{align*}
	which is a contradiction. Thus we obtain $\Phi\leq0$, as required.
\end{proof}

\begin{lemma}
\label{prop:M-T-ineq}
Denote
\begin{align*}
	A_s\coloneqq A_{\eps,s}\coloneqq\int_{Y\times[0,T]}(-u^\eps-s)^+ e^{F_\eps}\, \d V\d t
	=\int_{Y\times[0,T]}(-u^\eps-s)^+ e^{-u^\eps+L_\eps(u^\eps)} g_\eps \d V \d t.
\end{align*}
There exists a dimensional constant $c_n q>0$ such that for any $s\geq\|u_0\|_{L^\infty}$,
	\begin{align}
	\label{eq:M-T-bound}
		\sup_{t \in [0,T]} \int_Y\exp\left(c_n q A_{s}^{-\frac{1}{n+1}}((-u_t^\eps-s)^+)^{\frac{n+2}{n+1}} \right)\d V
		\leq K,
	\end{align}
    for some uniform constant $K$  (independent of $\varepsilon$) .
\end{lemma}

\begin{proof}
     Fix $q\in(1,2)$. By Lemma~\ref{lem:Phi_leq0} and \eqref{eq:unif-skoda-psi_j}, one has
     \begin{align}
     \label{prop:M-T-ineq eqn 1}
     	\int_Y\exp\left(c_n q A_{j,s}^{-\frac{1}{n+1}}((-u^\eps-s)^+)^{\frac{n+2}{n+1}} \right)\d V
     	\leq e^{C_n q A_{j,s}} \int_Y e^{-q\psi_j}\, \d V
     	\leq Ke^{KA_{j,s}}.
     \end{align}
     By $u^{\eps}\leq u_{0}$, one has
     $$A_s\leq 
     \int_{Y\times[0,T]}(-u^\eps)^+e^{F_\eps}\,\d V\d t
     \leq\int_{Y\times[0,T]}(\|u_{0}\|_{L^{\infty}}-u^\eps)e^{F_\eps}\,\d V\d t.
      $$
      Letting $j\to\infty$ in \eqref{prop:M-T-ineq eqn 1} and using $A_{j,s}\rightarrow A_s$, we obtain
\begin{align}
	\label{eq:M-T-ineq}
		\int_Y\exp\left(c_n q A_{s}^{-\frac{1}{n+1}}((-u^\eps-s)_+)^{\frac{n+2}{n+1}} \right)\d V
		\leq K \exp\left( {K}\int_{Y\times[0,T]}(\|u_{0}\|_{L^{\infty}}-u^\eps)e^{F_\eps}\,\d V\d t \right).
	\end{align}
Applying the elementary inequality $xy\leq x\log x+e^{y-1}$ for $x,y>0$, one has
\begin{align*}
	\int_{[0,T]\times Y}(\|u_{0}\|_{L^{\infty}}-u^\eps) e^{F_\eps}\, \d V\d t
	= {} &\int_{Y\times[0,T]}q(\|u_{0}\|_{L^{\infty}}-u^\eps)\cdot q^{-1}e^{F_\eps}\, \d V\d t \\
	\leq {} & \int_{Y\times[0,T]}e^{-qu^\eps+q\|u_{0}\|_{L^{\infty}}-1}\,\d V\d t
+\int_{Y\times[0,T]}\log(q^{-1}e^{F_\eps})q^{-1}e^{F_\eps}\,\d V\d t \\
	\leq {} & TK+q^{-1}\int_{Y\times[0,T]}(F_\eps+1) e^{F_\eps}\,\d V\d t 
	\leq K,
\end{align*}
where we used Proposition \ref{prop: uniform_skoda_eps} and Lemma \ref{eq:unif-Lp} in the last line. Therefore, together with \eqref{eq:M-T-ineq}, we have \eqref{eq:M-T-bound}.
\end{proof}

\begin{lemma}
	\label{lem:unif_De Giorgi}
	Denote 
	\begin{align*}
		\phi(s)\coloneqq\int_{Y \times [0,T] \cap \{u^\eps<-s\}}e^{F_\eps}\,\d Vdt
 	=\int_{Y \times [0,T] \cap \{u^\eps<-s\}} e^{-u^\eps+L_\eps(u^\eps)} g_\eps\, \d V dt.
	\end{align*}
    For any $r>n+1$, set $\delta\coloneqq\frac{1}{n+1}-\frac{1}{r}$, there exists a uniform constant $K_{\delta}$  (independent of $\varepsilon$) such that, for any $a>0$, $s\geq s_0:=\|u_0\|_{L^\infty}$,
	\begin{align*}
		a\phi(s+a)\leq K_{\delta}\phi(s)^{1+\delta}.
	\end{align*}
\end{lemma}
\begin{proof}
	Applying the elementary inequality $x^re^y\leq e^y(1+y)^r+C(r)e^{2x}$ (see \cite[Lemma~2.4]{CC23}) with
	\begin{align*}
		x=\frac{c_n q}{2}A_{s}^{-\frac{1}{n+1}}((-u^\eps-s)^+)^{\frac{n+2}{n+1}}, \quad
		y=\log(e^{F_\eps}+1),
	\end{align*}
	we have
	\begin{align*}
		& \left(\frac{c_n q}{2}\right)^rA_{s}^{-\frac{r}{n+1}}((-u^\eps-s)^+)^{r\frac{n+2}{n+1}}(e^{F_\eps}+1) \\
		\leq {} & (e^{F_\eps}+1)\big(1+\log(e^{F_\eps}+1)\big)^r
		+C(r)e^{c_n q A_{s}^{-\frac{1}{n+1}}((-u^\eps-s)^+)^{\frac{n+2}{n+1}} }.
	\end{align*}
	Integrating both sides on $Y\times[0,T]$, and using Lemma \ref{eq:unif-Lp}  and Lemma \ref{prop:M-T-ineq}, 
	\begin{align*}
		\int_{Y\times[0,T]}((-u^\eps-s)^+)^{r\frac{n+2}{n+1}}e^{F_\eps}\d V\d t
		\leq K A_{s}^{\frac{r}{n+1}}.
	\end{align*}
	Combining this with H\"older's inequality,
	\begin{align*}
		A_s
		= {} &\int_{Y\times[0,T]}(-u^\eps-s)^+ e^{F_\eps} \,\d V\d t \\
		\leq {} &\left( \int_{Y\times[0,T]}((-u^\eps-s)^+)^{r\frac{n+2}{n+1}} e^{F_\eps} \,\d V\d t\right)^{\frac{n+1}{r(n+2)}}\cdot
		\left( \int_{\{u^\eps<-s\}}e^{F_\eps} \,\d V\d t\right)^{1-\frac{n+1}{r(n+2)}} \\
		\leq {} & K^{\frac{n+1}{r(n+2)}} A_{s}^{\frac{1}{n+2}}\left( \int_{\{u^\eps<-s\}}e^{F_\eps} \,\d V\d t\right)^{1-\frac{n+1}{r(n+2)}},
	\end{align*}
	which implies
	\begin{align*}
		A_s\leq K^{\frac{1}{r}}\left( \int_{\{u^\eps<-s\}}e^{F_\eps}\, \d V\d t\right)^{\frac{n+2}{n+1}-\frac{1}{r}}.
	\end{align*}
	In addition, by the definition of $A_s$, one has
	\begin{align*}
		A_s\geq a\int_{\{u^\eps<-s-a\}}e^{F_\eps}\, \d V\d t.
	\end{align*}
	Then combining the above inequalities, we finish the proof.
\end{proof}

To obtain our uniform $C^0$ estimate we will use to following well known De Giorgi type iteration lemma:

\begin{lemma}[De Giorgi]
\label{lem: De_Giorgi_it}Let $\phi:\RR^+\rightarrow\RR^+$ be a decreasing function with constants $s_{0},q,\delta, B>0$ such that
\begin{align}
\label{eq: De_Giorgi_it-ineq}
r^q \phi(s+r)\leq B \phi(s)^{1+\delta} \quad \text{for $s\geq  s_0$ and $r>0$}.
\end{align}
Then $\phi(s)\equiv0$ for $s\geq 2^{1+\frac1\delta}B^{\frac1q} \phi(s_0)^{\frac{\delta}q}+s_0$.
\end{lemma}

\begin{proof}
We introduce the quantities
	\begin{align*}
		d \coloneqq 2^{1+\frac{1}{\delta}}B^{\frac{1}{q}} \phi(s_0)^{\frac{\delta}{q}},
        \qquad s_j \coloneqq s_0+d(1-2^{-j}).
	\end{align*}
We claim $\phi(s_j)\leq \phi(s_0)2^{-\frac{qj}\delta}$. Given this claim, by letting $j\to\infty$, one obtains $\phi(s_0+d)=0$ and then $\phi(s)\equiv0$ for $s\geq d+s_0$. Next we prove the claim by induction argument. It holds for $j=0$. Suppose it holds for $j$.
	Applying \eqref{eq: De_Giorgi_it-ineq} for $s=s_j$ and $r=s_{j+1}-s_j=d2^{-j-1}$, one has
	\begin{align*}
		\phi(s_{j+1})
		\leq d^{-q}2^{q(j+1)}B\phi(s_j)^{1+\delta}
		\leq d^{-q}2^{q(j+1)}B\phi(s_0)^{1+\delta}2^{-\frac{qj}{\delta}(1+\delta)}
		=\phi(s_0)2^{-\frac{q(j+1)}\delta}.
	\end{align*}
\end{proof}

Now we are ready to prove Proposition \ref{prop:unif_Linfty} and Proposition \ref{prop:sub_stab_holder}.

\begin{proof}[Proof of Proposition \ref{prop:unif_Linfty}]
Recall that $u^{\eps}\leq u_{0}$. It then suffices to establish the lower bound of $u^{\eps}$. Fix $r>n+1$ and $\delta=\frac{1}{n+1}-\frac{1}{r}$. Note that $\phi(s)$ is decreasing with respect to $s>0$.
By Lemma \ref{lem:unif_De Giorgi} and Lemma \ref{lem: De_Giorgi_it} with $q=1$ and $B=K_\delta$, we have
\begin{align*}
	\phi(s)\equiv0 \quad\text{ for }s\geq2^{1+\frac{1}{\delta}}K_\delta\phi(s_0)^{\delta}+s_0,
\end{align*}
where $s_0=\|u_0\|_{L^\infty}$. Together with
\begin{align*}
	\phi(s_0)=\int_{\{u^\eps<-s_0 \}}e^{-u^\eps+L_\eps(u^\eps)} g_\eps \, \d V \d t
	\leq \int_{Y\times[0,T]}e^{-u^\eps+L_\eps(u^\eps)} g_\eps\ \d V \d t=T,
\end{align*}
and using the continuity of $u^{\eps}$ and the strict positivity of $e^{F_\eps}$, we obtain the pointwise lower bound of $u^{\eps}$:
\begin{align*}
	u^\eps\geq -2^{1+\frac{1}{\delta}}K_\delta T^{\delta}-\|u_0\|_{L^\infty}.
\end{align*}
\end{proof}

\begin{proof}[Proof of Proposition \ref{prop:sub_stab_holder}]
The stability estimate \eqref{eq:quasi-decreasing} and time H\"older estimate \eqref{eq:time Holder} can be proved by modifying \cite[Theorem~2.1 and Proposition 2.5]{CC23} and using Lemma~\ref{lem: uniform_skoda}. The argument is analogous to that of Proposition \ref{prop:unif_Linfty}  via the suitable auxiliary flow. We provide a sketch of the proof here and explain the explicit uniform constant in our case.

\medskip
\noindent
{\bf Step 1:} Let $\delta\in(0,1)$ and $v\in C^\infty(Y\times[0,T])$ such that $\dot v\leq0$ and $\omega_{\eps'}+\ddc v\geq-\frac{\delta}{2}\omega$. Let $s>0$ and consider the function $(1-\delta)v-u^{\eps'}-s$.
We consider the following auxiliary flow:
\begin{equation}\label{eq:auxiliary_equ-stability}
\begin{cases}
V_{\omega_{\eps'}}^{-1}(-\dot{\psi}_{j,\delta})(\omega_{\eps'}+\ddc{\psi}_{j,\delta})^n
 	=A_{j,\delta,s}^{-1}\cdot\eta_j\Big((1-\delta)v-u^{\eps'}-s\Big) e^{F_{\eps'}}\,\d V
& \mbox{in $Y\times[0,T]$},  \\[1mm]
\psi_{j,\delta}(t)\in\PSH_{\omega_{\eps'}} & \mbox{on $Y\times\{t\}$}, \\[1mm]
\psi_{j,\delta}  = 0 & \mbox{on $Y\times\{0\}$},
\end{cases}
\end{equation}
where 
\begin{align*}
    \psi_{j,\delta}\coloneqq\psi_{\eps',j,\delta},
    \quad
	A_{j,\delta,s}\coloneqq A_{\eps',j,\delta,s}\coloneqq \int_{Y\times [0,T]}\eta_j((1-\delta)v-u^{\eps'}-s) e^{F_{\eps'}}\d V \d t.
\end{align*}
Let $s>\|((1-\delta)v_0-u_0^{\eps'})^+\|_{L^\infty}$, then there exist dimensional constants $c_n>0$ and $C_n>0$ such that
	\begin{align}
	\label{eq:Phi-delta-leq0}
		c_n A_{j,\delta,s}^{-\frac{1}{n+2}}((1-\delta)v-u^{\eps'}-s)
		\leq (-\psi_{j,\delta}+C_n\delta^{-(n+2)} A_{j,\delta,s})^{\frac{n+1}{n+2}} \ \ \ \textup{ on } Y \times [0,T].
	\end{align}
	The proof of \eqref{eq:Phi-delta-leq0} is very similar to the proof of Lemma~\ref{lem:Phi_leq0} (cf. \cite[Lemma~2.6]{CC23}).

\medskip
\noindent
{\bf Step 2:}
There exists a dimensional constant $q_n >0$ such that for any $s\geq\|((1-\delta)v_0-u_0)^+\|_{L^\infty}$ and $0<\delta<1$,
	\begin{align}
	\label{eq:M-T-bound-delta}
		\sup_{t \in [0,T]} \int_Y\exp\left(q_n A_{\delta,s}^{-\frac{1}{n+1}}\Big(((1-\delta)v-u_t^{\eps'}-s)^+\Big)^{\frac{n+2}{n+1}} \right)\d V
		\leq \exp(K\delta^{-(n+2)}A_{\delta,s}),
	\end{align}
where
\begin{align*}
	A_{\delta,s}\
	\coloneqq A_{\eps',\delta,s}\coloneqq&\int_{Y\times[0,T]}((1-\delta)v-u^{\eps'}-s)^+ e^{F_{\eps'}}\, \d V\d t \\
	=&\int_{Y\times[0,T]}((1-\delta)v-u^{\eps'}-s)^+ e^{-u^{\eps'}+L_\eps(u^{\eps'})} g_{\eps'} \d V \d t.
\end{align*}
Indeed, by \eqref{eq:Phi-delta-leq0} and uniform integrability Lemma~\ref{lem: uniform_skoda}, we have for $q>1$
     \begin{align*}
     	\int_Y\exp\left(qc_n^{\frac{n+2}{n+1}}  A_{j,\delta,s}^{-\frac{1}{n+1}}(((1-\delta)v-u^{\eps'}-s)^+)^{\frac{n+2}{n+1}} \right)\d V
     	\leq e^{C_n q\delta^{-(n+2)} A_{j,\delta,s}} \int_Y e^{-q\psi_{j,s}}\, \d V
     	\leq Ke^{K\delta^{-(n+2)} A_{j,\delta,s}}.
     \end{align*}
     Letting $j\rightarrow\infty$ and $q_n\coloneqq qc_n^{\frac{n+2}{n+1}}$, we obtain \eqref{eq:M-T-bound-delta}.

\medskip
\noindent
{\bf Step 3:} Let $s_0$ be chosen such that $s_0\geq\|((1-\delta)v_0-u_0^{\eps'})^+\|_{L^\infty}$ and $A_{\delta,s_0}\leq \delta^{n+2}$. Then for any $p>1$,
\begin{align*}
	\|((1-\delta)v-u^{\eps'}-s)^+\|_{L^p(\d V\d t)}
	\leq K(p)A_{\delta,s}^{\frac{1}{n+2}},
\end{align*} 
where $K(p)=K(T,p,q_n)$. 

This is an immediate application of Step 2. 

\medskip
\noindent
{\bf Step 4:} For any $0<\eta<\frac{1}{p_0^*(n+1)}$, there exists a uniform constant $B_0>0$ such that for any $r>0$ and $s\geq s_0$ ($s_0$ is defined in Step 3),
\begin{align*}
	r\vol(\Omega_{\delta,s+r})
	\leq B_0\vol(\Omega_{\delta,s})^{1+\eta},
\end{align*}
where $\Omega_{\delta,s}\coloneqq\{(1-v)v-u^{\eps'}-s>0\}$.

The proof is exactly same as \cite[Lemma~2.8]{CC23}. The constant $\eta$ and$B_0$ can be computed as
\begin{align*}
1+\eta=\Big(1+\frac{1}{p_0^*(n+1)}\Big)(1-\gamma^{-1}), \quad
	B_0=K(\gamma)K(p_0^*\gamma)^{\frac{1}{n+1}}\|e^{F_{\eps'}}\|_{L^{p_0}}^{\frac{1}{n+1}},
\end{align*}
where $\gamma>1$ is arbitrary and $K(\gamma)$ is the constant in Step 3.

\medskip
\noindent
{\bf Step 5:} Applying De Giorgi type Lemma \ref{lem: De_Giorgi_it} to Step 4, we obtain
\begin{align*}
	\sup_{Y\times[0,T]}((1-\delta)v-u^{\eps'})
	\leq s_0+K_1\vol(\Omega_{\delta,s})^\eta,
\end{align*}
where $K_1=\frac{2B_0}{1-2^{-\eta}}$.
Furthermore, if $s_0\geq 2\delta\|v\|_{L^\infty}$, then
\begin{align*}
	\vol(\Omega_{\delta,s_0})
	\leq {} &\frac{1}{s_0}\int_{\Omega_{\delta,s_0}}((1-\delta)v-u^{\eps'})^+\d V\d t \\
	\leq {} &\frac{1}{s_0}\Big(\|(v-u^{\eps'})^+\|_{L^1}+\delta\|v\|_{L^\infty}\vol(\Omega_{\delta,s_0}) \Big) 
	\leq \frac{2}{s_0}\|(v-u^{\eps'})^+\|_{L^1}.
\end{align*}
Therefore, if $s_0\geq\|((1-\delta)v_0-u_0)^+\|_{L^\infty}$, $A_{\delta,s_0}\leq \delta^{n+2}$ and $s_0\geq 2\delta\|v\|_{L^\infty}$, we have
\begin{align}
\label{eq:stability-partial}
	\sup_{Y\times[0,T]}(v-u^{\eps'})
	\leq\frac{3}{2} s_0+K_2s_0^{-\eta}\|(v-u^{\eps'})^+\|_{L^1}^\eta.
\end{align}

\medskip
\noindent
{\bf Step 6:} Given $\delta\in(0,1)$, one defines
\begin{align*}
	s_*(\delta)\coloneqq\{s_0 \text{ satisfies the three conditions in Step 5} \}.
\end{align*}
Then
\begin{align}
\label{eq:stability-partial-inf}
	\sup_{Y\times[0,T]}(v-u^{\eps'})
	\leq\inf_{s_0\geq s_0(\delta)}\left(\frac{3}{2} s_0+K_2s_0^{-\eta}\|(v-u^{\eps'})^+\|_{L^1}^\eta\right).
\end{align}
Given $\gamma>1$, there exist a constant $K_1(\gamma)>0$ such that
\begin{align}
	\label{eq:estimate-s_*(delta)}
	s_*(\delta)
	\leq\max\left\{2\|(v_0-u_0)^+\|_{L^\infty}, \, 2\delta\|v\|_{L^\infty},\, 
	K_1(\gamma)\delta^{-\frac{p_0^*(n+1)}{1-\beta^{-1}}}\|(v-u^{\eps'})^+\|_{L^1} \right\}.
\end{align}
The proof of \eqref{eq:estimate-s_*(delta)} is exactly same as \cite[Lemma~2.9]{CC23}. In our case, 
\begin{align*}
	K_1(\gamma)=2\|e^{F_{\eps'}}\|_{L^{p_0}}^{\frac{p_0^*}{1-\gamma^{-1}}}K(\gamma p_0^*)^{\frac{p_0^*}{1-\gamma^{-1}}},
\end{align*}
where $K(\gamma p_0^*)$ is the constant defined in Step 3.

\medskip
\noindent
{\bf Step 7:} applying $v=u^\eps$, then $\omega_{\eps'}+\ddc u^\eps\geq (\eps'-\eps)\omega_Y\geq0$ and $v_0=u_0^\eps=u_0^{\eps'}$. We prove \eqref{eq:quasi-decreasing} by adapting the method in \cite[Theorem~2.1]{CC23}.
We may assume that $\|(u^{\eps}-u^{\eps'})^+\|_{L^1}<1$. 
We take $\delta=\|(u^{\eps}-u^{\eps'})^+\|_{L^1}^{\frac{1}{1+p_0^*(n+1)}}$ and $\gamma\gg1$ such that $1-\frac{p_0^*(n+1)}{1+p_0^*(n+1)}\frac{1}{1-\gamma^{-1}}>\beta$, where $\beta\in(0,\frac{1}{1+p_0^*(n+1)})$ is the constant in the hypothesis of Proposition~\ref{prop:sub_stab_holder}.
We define $K_3\coloneqq\max\{2\|u^\eps\|_{L^\infty}, K_1(\gamma)\}$.
We claim:
\begin{align}
\label{eq:special-s_0}
	s_0\coloneqq K_3\|(u^{\eps}-u^{\eps'})^+\|_{L^1}^\beta\geq s_*(\delta)
	=\max\left\{2\delta\|u^\eps\|_{L^\infty},\, 
	K_1(\gamma)\delta^{-\frac{p_0^*(n+1)}{1-\beta^{-1}}}\|(u^\eps-u^{\eps'})^+\|_{L^1} \right\}.
\end{align}
This is clear from our choice for $\delta$, $K_3$ and $\gamma$.

Once the Claim holds. By \eqref{eq:stability-partial-inf} and choosing $\eta$ such that $\beta=\frac{\eta}{1+\eta}$,
\begin{align*}
	\sup_{Y\times[0,T]}(v-u^{\eps'})
	\leq {} &\frac{3}{2} K_3\|(u^{\eps}-u^{\eps'})^+\|_{L^1}^\beta
	+K_2K_3^{-\eta}\|(u^{\eps}-u^{\eps'})^+\|_{L^1}^{-\beta\eta}\|(v-u^{\eps'})^+\|_{L^1}^\eta \\
	\leq {} &\Big(\frac{3}{2}K_3+K_2K_3^{-\eta}\Big)\|(u^{\eps}-u^{\eps'})^+\|_{L^1}^\beta.
\end{align*}
We obtain \eqref{eq:quasi-decreasing} by letting $K_\beta=\frac{3}{2}K_3+K_2K_3^{-(1-\beta)^{-1}+1}$.

\medskip
\noindent
{\bf Step 8:} Once \eqref{eq:quasi-decreasing} is proved. By applying the local time-integration technique and \cite[Lemma~2.10]{CC23}, we immediately obtain \eqref{eq:time Holder}. The proof is exactly same as \cite[Prop~2.5, see page~2891]{CC23}.
\end{proof}

Next we turn our attention to a local uniform $t$-Lipschitz estimate. The key idea is to compare the evolving metric with a time independent metric whose volume form is proportional to $g_\eps\,\d V$.

\begin{proposition}[local $t$-Lipschitz estimate]
\label{prop:local_t-deriv_upper}For all $\tau \in (0,T)$, there is a uniform constant $K_{\tau,T}$ (independent of $\varepsilon$) such that
\[
0<(-\dot u^\eps)\leq K_{\tau,T}
\quad\text{on }Y\times[\tau,T].
\]
\end{proposition}
\begin{proof}
Fix $T>0$. By Proposition~\ref{prop:unif_Linfty}, there is a uniform
constant $U_T>0$ such that
\[
\|u^\eps\|_{L^\infty(Y\times[0,T])}\leq U_T.
\]
We set  $m_\eps:=\int_Yg_\eps\,\d V$.
Since $g_\eps\to g$ in $L^{p_0}(Y,\d V)$, with $p_0>1$ and
$\int_Yg\,\d V>0$, the numbers $m_\eps$ are uniformly bounded above
and away from zero, and $\|g_\eps\|_{L^{p_0}}$ is uniformly bounded, by construction.
By Yau's theorem~\cite{Yau78a}, there exits 
$\varphi_\eps\in\cH_{\omega_\eps}$ satisfying
\begin{equation}\label{eq:elliptic_reference_metric}
(\omega_\eps+\ddc\varphi_\eps)^n=\frac{V_{\omega_\eps}}{m_\eps}g_\eps\,\d V,
\qquad
\sup_Y\varphi_\eps=0.
\end{equation}
Since $\omega_\varepsilon$ converges smoothly to the semipositive
and big form $\pi^*\omega$, and the densities
$(V_\varepsilon/m_\varepsilon)g_\varepsilon$ are uniformly bounded
in $L^{p_0}(Y,dV)$, the uniform version of Ko{\l}odziej's
$L^\infty$ estimate \cite[Theorem~4.1 and Remark~4.5]{BEGZ10} (c.f. \cite{Kol98})
gives
\begin{equation}\label{eq:elliptic_reference_bound}
-C_\varphi\leq\varphi_\eps\leq0
\end{equation}
for a constant $C_\varphi$ independent of $\eps$.

For the calculations below, we fix $\eps>0$ and we abbreviate
\[
u:=u^\eps,\quad w:=-\dot u>0,\quad
L:=L_\eps(u), \quad \omega':=\omega_\eps+\ddc u,\quad  \varphi:=\varphi_\eps,\quad
\widehat\omega:=\omega_\eps+\ddc\varphi_\eps.
\]
Dividing \eqref{eq:elliptic_reference_metric} by the flow equation
\eqref{eq:eps_IMAF T} cancels $g_\eps$ and yields
\[
\frac{\widehat\omega^n}{\omega'^n}
=\frac{w e^{u-L}}{m_\eps}.
\]
Since
\[
m_\eps e^{-U_T}\leq e^{-L}
=\int_Ye^{-u}g_\eps\,\d V\leq m_\eps e^{U_T},
\]
we have $e^{u-L}/m_\eps\geq e^{-2U_T}$. Hence, the arithmetic--geometric mean
inequality gives
\begin{equation}\label{eq: elliptic_ref_trace}
\tr_{\omega'}\widehat\omega\geq n\left(\frac{\widehat\omega^n}{\omega'^n}\right)^{1/n}
\geq c_Tw^{1/n},
\qquad c_T:=n e^{-2U_T/n}>0.
\end{equation}
Differentiating $L$ and the logarithm of \eqref{eq:eps_IMAF T} gives
\begin{equation}\label{eq: speed_evo}
\dot L=-\int_Yw e^{-u+L}g_\eps\,\d V\leq0,
\qquad
\cP w=w(w+\dot L),
\qquad
\cP:=\partial_t-w\Delta_{\omega'}.
\end{equation}
Consider the bounded positive potential difference
\[
B:=1+U_T+u-\varphi,
\qquad
1\leq B\leq B_T:=1+2U_T+C_\varphi,
\]
and the quotient $Q:=w/B$. Since
$\partial_tB=-w$ and $\Delta_{\omega'}B=n-\tr_{\omega'}\widehat\omega$, we obtain
\[
\cP B=w\big(\tr_{\omega'}\widehat\omega-(n+1)\big).
\]
Define $M(t):=\max_Y Q(\cdot,t)\coloneqq Q(z_t,t)$ for some point $z_t\in Y$ and $t>0$.
At $z_t$, its gradient vanishes. Applying the
product rule to $w=BQ$ at $z_t$, and then using
\eqref{eq: speed_evo}, we find
\begin{equation}\label{eq: speed_quot_evo}
\cP Q
=\frac{\cP w-Q\cP B}{B}
=Q\dot L+Q^2(B+n+1-\tr_{\omega'}\widehat\omega)
\leq A_TQ^2-c_TQ^{2+1/n},
\end{equation}
where $A_T:=B_T+n+1$ and we have used $\tr_{\omega'}\widehat\omega\geq c_T(BQ)^{1/n}\geq c_TQ^{1/n}$, due to
 \eqref{eq: elliptic_ref_trace} and $B\geq1$.
Since $M(t)$ is locally Lipschitz, Hamilton's trick (as in the proof of Proposition \ref{prop:alpha_IMAF_C0_est}) and \eqref{eq: speed_quot_evo} shows that for almost every $t$,
\begin{equation}\label{eq:speed_maximum_ode}
M'(t)\leq A_TM(t)^2-c_TM(t)^{2+1/n}.
\end{equation}
Choose $K_T:=\max\{1,(2A_T/c_T)^n\}$. When $M\geq K_T$,
\eqref{eq:speed_maximum_ode} implies
$M'\leq-(c_T/2)M^{2+1/n}$.
Thus $\ell:=(M-K_T)^+$ satisfies
\[
\ell'\leq-\frac{c_T}{2}\ell^{2+1/n}
\quad\text{for almost every }t>0.
\]
Integrating this scalar differential inequality gives
\[
\ell^{-1-1/n}(t) \geq \ell^{-1-1/n}(0)+\frac{c_T(n+1)}{2n}t
\geq \frac{c_T(n+1)}{2n}t, \qquad 0<t\leq T
\]
and so
\[
M(t) \leq K_{T}+\ell(t)
\leq K_T+
\left(\frac{c_T(n+1)}{2n}t\right)^{-n/(n+1)},
\qquad 0<t\leq T.
\]
Since $w=BQ\leq B_TM$, we conclude that
\begin{equation}\label{eq:explicit_speed_bound}
0<-\dot u^\eps(x,t)\leq C_Tt^{-n/(n+1)}
\quad\text{for }(x,t)\in Y\times(0,T].
\end{equation}
The estimate of the proposition follows with $K_{\tau,T}:=C_T\tau^{-n/(n+1)}$.
\end{proof}

\begin{remark}\label{rem:improved_time_holder}
Estimate \eqref{eq:explicit_speed_bound} from the above proof also allows to improve the H\"older estimate \eqref{eq:time Holder}. Indeed, after integration, for $0\leq s<t\leq T$,
we obtain
\begin{equation}\label{eq:improved_time_holder}
|u^\eps(x,t)-u^\eps(x,s)|
\leq (n+1)C_T\left(t^{1/(n+1)}-s^{1/(n+1)}\right)
\leq (n+1)C_T|t-s|^{1/(n+1)}.
\end{equation}
The constant is independent of $\eps>0$. In particular, this gives an
alternative proof of the time H\"older estimate \eqref{eq:time Holder}, without using the stability estimate \eqref{eq:quasi-decreasing}.
\end{remark}

Next we move on to second order estimates.

\begin{proposition}[local $C^{2}$ estimate]
\label{prop:logtr_omega_Y}
Fixing $\varepsilon >0$ and denoting $\omega'\coloneqq \omega_\eps+\ddc{u}^\eps,$ there is a uniform constants $C>0$  (independent of $\varepsilon$) such that
\[
\log\tr_{\omega_Y}\omega'
	\leq -\psi_\eps^- -C\psi+C.
\]
\end{proposition}

\begin{proof}
We consider the test function $Q\coloneqq\log\tr_{\omega_Y}\omega'+\psi_\eps^- -A(u^\eps-\psi)-t$, where $A>0$ determined later. Suppose $Q$ attains its maximum on $Y \times [0, T]$ at a point $(x_0, t_0)$ with $t_0 > 0$. Note that $Q\to-\infty$ along $E_{0}$. Then $x_{0}\in Y\backslash E_{0}$. 
By Proposition \ref{prop:unif_Linfty}, it suffices to show $Q(x_{0},t_{0})\leq C$. Define the parabolic operator:
\begin{align*}
	\cL_\eps\coloneqq\Delta_{\omega'}-(-\dot{u}^\eps)^{-1}\partial_t.
\end{align*}
By the maximum principle, at $(x_{0},t_{0})$, one has
\begin{equation}\label{prop:logtr_omega_Y eqn 1}
0 \geq \cL_{\eps}(Q) = \cL_{\eps}(\log\tr_{\omega_Y}\omega'+\psi_{\eps}^{-})
-A\cL_{\eps}(u^{\eps}-\psi)
+(-\dot{u}^{\eps})^{-1}.
\end{equation}
By Yau's Schwarz lemma \cite{Yau78b}  (see also e.g. \cite[Chap 3, Lemma~3.7]{Sze14}),
\begin{align}
\label{eq:Yau-inequ}
	\Delta_{\omega'} \log\tr_{\omega_Y}\omega'
	\geq -B\tr_{\omega'}\omega_Y-\frac{\tr_{\omega_Y}\Ric(\omega')}{\tr_{\omega_Y}(\omega')},
\end{align}
where $-B$ is a lower bound for the bisectional curvature of $\omega_Y$. Taking the logarithm of both sides of \eqref{eq:eps_IMAF T}, one has
\begin{align*}
	\log V_\eps^{-1}+\log(-\dot{u}^\eps) +\log\omega'^n
	=-u^\eps+L_\eps(u^\eps)+\psi_\eps^+-\psi_\eps^-+\log \d V,
\end{align*}
and then
\begin{align}
	\tr_{\omega_Y}\Ric(\omega')
	= {} & \tr_{\omega_Y}(-\ddc\log\omega'^n)  \nonumber \\[3mm]
	= {} & \Delta_{\omega_Y}\left(\log(-\dot{u}^\eps)+u^\eps
    -\psi_{\eps}^{+}+\psi_{\eps}^{-}\right)+\tr_{\omega_Y}\Ric(\d V) \nonumber \\[1mm]
	= {} &-\frac{\Delta_{\omega_Y}\dot{u}^\eps}{(-\dot{u}^\eps)}-\frac{|\nabla\dot{u}^\eps|_{\omega_Y}^2}{(-\dot{u}^\eps)^2}+\tr_{\omega_Y}\omega'
    -\tr_{\omega_Y}\omega_\eps 
	  -\Delta_{\omega_Y}\psi_{\eps}^{+}+\Delta_{\omega_Y}\psi_{\eps}^{-}
      +\tr_{\omega_Y}\Ric(\d V) \nonumber \\
	\leq {} &-\frac{\Delta_{\omega_Y}\dot{u}}{(-\dot{u}^\eps)}+\tr_{\omega_Y}\omega'+\Delta_{\omega_Y}\psi_\eps^- +C_1 \nonumber \\
	= {} &-\frac{\p_t\tr_{\omega_Y}\omega'}{(-\dot{u}^\eps)}+\tr_{\omega_Y}\omega'+\Delta_{\omega_Y}\psi_\eps^-+C_1,
	\label{eq:tr_omega_Y(Ric')}
\end{align}
where $C_1>0$ is a uniform constant such that $-\Delta_{\omega_Y}\psi_\eps^++\tr_{\omega_Y}\Ric(\d V)\leq C_1$ and we have used $\Delta_{\omega_Y}\dot{u}^\eps=\p_t\tr_{\omega_Y}\omega'$ for the last equality. Since $B_0\omega_Y+\ddc\psi_\eps^-\geq0$, then 
\begin{align*}
	0\leq B_0\omega_Y+\ddc\psi_\eps^-
	\leq \tr_{\omega'}(B_0\omega_Y+\ddc\psi_\eps^-)\omega'.
\end{align*}
Applying $\tr_{\omega_Y}$ on both sides,
\begin{align*}
	0\leq \tr_{\omega_Y}(B_0\omega_Y+\ddc\psi_\eps^-)
	\leq \tr_{\omega'}(B_0\omega_Y+\ddc\psi_\eps^-)\tr_{\omega_Y}\omega'.
\end{align*}
and so
\begin{align}
\label{eq:Lap_psi-}
	\Delta_{\omega_Y}\psi_\eps^-
	\leq \tr_{\omega'}(B_0\omega_Y+\ddc\psi_\eps^-)\tr_{\omega_Y}\omega'.
\end{align}
Combining \eqref{eq:Yau-inequ}, \eqref{eq:tr_omega_Y(Ric')} and \eqref{eq:Lap_psi-},
then we obtain
\begin{align*}
	\Delta_{\omega'}\Big( \log\tr_{\omega_Y}\omega'+\psi_\eps^-\Big)
	\geq -(B+B_0)\tr_{\omega'}\omega_Y +(-\dot{u}^\eps)^{-1}\p_t\log\tr_{\omega_Y}\omega'
	-1-\frac{C_1}{\tr_{\omega_Y}\omega'}.
\end{align*}
We may assume that $\tr_{\omega_Y}\omega'\geq 1$. Otherwise, we are done. It follows that
\begin{equation}\label{prop:logtr_omega_Y eqn 2}
	\cL_\eps\Big( \log\tr_{\omega_Y}\omega'+\psi_\eps^-\Big)
	\geq -(B+B_0)\tr_{\omega'}\omega_Y 
	-(1+C_1).
\end{equation}
On the other hand, by $\omega_{\eps}+\ddc\psi\geq(1+\eps)\omega_{Y}$,
\begin{equation}\label{prop:logtr_omega_Y eqn 3}
\cL_\eps(u^\eps-\psi) = \Delta_{\omega'}(u^\eps-\psi)+1 
= n+1-\tr_{\omega'}(\omega_{\eps}+\ddc\psi) 
\leq n+1-(1+\eps)\tr_{\omega'}\omega_{Y}.
\end{equation}
Substituting \eqref{prop:logtr_omega_Y eqn 2} and \eqref{prop:logtr_omega_Y eqn 3} into \eqref{prop:logtr_omega_Y eqn 1},
\begin{align*}
0 \geq -(B+B_0)\tr_{\omega'}\omega_Y -(1+C_1)
	 -A(n+1)+A(1+\eps)\tr_{\omega'}\omega_Y+(-\dot u^\eps)^{-1}.
\end{align*}
By choosing $A\coloneqq B+B_0+1$, we have
\begin{align*}
	\tr_{\omega'}\omega_Y \leq C 
    \quad\text{and}\quad 
	(-\dot u^\eps)^{-1} \leq C.
\end{align*}
By the equation \eqref{eq:eps_IMAF T}, at $(x_0,t_0)$,
\[
\frac{\omega'^n}{\omega_{Y}^{n}}
= V_\eps(-\dot{u}^\eps)^{-1}e^{-u^\eps+L_\eps(u^\eps)+\psi_{\eps}^{+}-\psi_{\eps}^{-}}\frac{\d V}{\omega_{Y}^{n}}
\leq Ce^{-\psi_\eps^{-}},
\]
It follows that
\begin{align}
\label{eq:tr_relation}
	\tr_{\omega_Y}\omega'
	\leq \frac{1}{(n-1)!}(\tr_{\omega'}\omega_Y)^{n-1}\frac{\omega'^n}{\omega_Y^n}
    \leq Ce^{-\psi_\eps^{-}}.
\end{align}
Combining this with Proposition \ref{prop:unif_Linfty}, $\sup_{Y}\psi_{\eps}^{-}\leq C$ and $\sup_{Y}\psi\leq C$, we obtain
\[
Q(x_{0},t_{0}) = \log\tr_{\omega_Y}\omega'+\psi_\eps^- -A(u^{\eps}-\psi)-t \leq C,
\]
as required.
\end{proof}

\bigskip
\begin{theorem}
	\label{thm:local-C2}
	For any $F\Subset Y\backslash E$ and $\tau\in(0,T)$, there is a uniform constant $K_{F,\tau,T}>0$  (independent of $\varepsilon$)  such that
	\begin{align}
		\label{eq:unif-ellip}
        K_{F,\tau,T}^{-1}\omega_Y\leq \omega_\eps+\ddc u^\eps
\leq K_{F,\tau,T}\omega_Y
\quad\text{on }F\times[\tau,T].
	\end{align}
	Furthermore, for any positive integer $k$ and some $0<\alpha<1$, 
	\begin{align}\label{eq: C^k_est}
		\|u^{\eps}\|_{C_x^{k}C_{t}^{1,\alpha/2}(F\times[\tau,T])} \leq K_{k,F,\tau,T,\alpha}.
	\end{align}
\end{theorem}

{

\begin{proof} We use again the abbreviation $\omega'\coloneqq \omega_\eps+\ddc{u}^\eps,$. The estimate \eqref{eq:unif-ellip} follows from Proposition \ref{prop:unif_Linfty}, Proposition~\ref{prop:local_t-deriv_upper}, Proposition~\ref{prop:logtr_omega_Y} and
\begin{align*}
	\tr_{\omega'}\omega_Y
	\leq \frac{1}{(n-1)!}(\tr_{\omega_Y} \omega')^{n-1}\frac{\omega_Y^n}{\omega'^n}
	=\frac{1}{(n-1)!}(\tr_{\omega_Y} \omega')^{n-1}
    V_{\omega_{\eps}}^{-1}(-\dot{u})
    e^{u^{\eps}-L_{\eps}(u^{\eps})-\psi_{\eps}^{+}+\psi_{\eps}^{-}}
    \frac{\omega_{Y}^{n}}{\d V}.
	\end{align*}
Then by using parabolic complex Evans--Krylov estimate \cite{Kry82,Chu16}, we obtain
\[
\|u^{\eps}\|_{C_{x}^{2,\alpha}C_{t}^{1,\alpha/2}(F\times[\tau,T])} \leq K_{F,\tau,T,\alpha}.
\]
Taking the spatial derivative of the on the both sides of \eqref{eq:eps_IMAF T}, and using parabolic complex Evans--Krylov estimate again and parabolic Schauder theory \cite{Lie96} (in the spatial direction), we conclude \eqref{eq: C^k_est}.
\end{proof}

\begin{remark}
\begin{enumerate} 
    \item There is an obstruction to applying the parabolic Evans--Krylov theory to obtain higher-order estimates in the time variable for $u^{\varepsilon}$. The reason is that the equation \eqref{eq:eps_IMAF T} contains the `purely global' term $L_\varepsilon(u^\varepsilon_t)$. The parabolic Evans--Krylov theory only gives the local estimate, which can not bound the higher $t$-derivatives of $L_\eps(u_t^\eps)$.
    
    \item If we take the following reparametrization with respect to time
\begin{align*}
    \tau_\eps(t)\coloneqq\int_0^te^{L_\eps(u^\eps(a))}da, \text{ and } U^\eps(x,\tau)\coloneqq
     U^\eps(x,\tau_\eps(t))
    =u^\eps(x,t),
\end{align*}
then
\begin{align*}
    \p_\tau U^{\eps}(x,\tau)
    =e^{-L_\eps(u^\eps_t)}\p_t u^\eps(x,t)
\end{align*}
and the original equation can be re-written as
\begin{align}
\label{eq:re-param-equation}
    V_\eps^{-1}(-\p_\tau U^\eps)(\omega_\eps+\ddc U^\eps)^n
    =e^{-U^\eps(x,\tau)}g_\eps \d V.
\end{align}
Now, the equation \eqref{eq:re-param-equation} is purely local without the `global term' involving the $L$ functional. Then applying the parabolic Evans--Krylov theory, we can obtain the local higher order estimate
\begin{align*}
    \|U^\eps\|_{C_x^{k}C_{t}^{k/2}(F\times[\tau,T])}\leq K_{k,F,\tau,T}.
\end{align*}
\end{enumerate}    
\end{remark}
}

\medskip
For fixed $\varepsilon>0$ let us recall our earlier notation:
\begin{align*}
	w_\eps(t)=(-\dot u^\eps(t))>0, \quad \omega'=\omega'(t)\coloneqq\omega_\eps+\ddc u^\eps(t),\quad\text{and} \quad
	\mu_{\omega'(t)}\coloneqq e^{-u^\epsilon(t)+L_\epsilon(u^\epsilon(t))}
g_\epsilon\, \d V.
\end{align*}

\begin{proposition}
\label{lem:La-density-eps}
    For any $0<T<\infty$, then
	there are uniform constants $\delta_0>0$ and $K_{a,T}>0$  (independent of $\varepsilon$)  such that for any $0<\delta<\delta_0$ with $a=1+\delta$, 
	\begin{align*}
		I_{\eps,a}(t)\coloneqq \int_Yw_\eps(t)^{-a}\,\mu_{\omega'(t)}
\leq K_{a,T}
\quad \text{for $t\in[0,T]$}.
	\end{align*}
\end{proposition}

\begin{proof}
	We first prove that $I_{\eps,a}(0)$ is uniformly bounded.
	One chooses sufficiently small $\delta_0>0$ such that
	\begin{align*}
		\int_Y e^{-\delta_0\psi^+}\d V<\infty.
	\end{align*} 
	Since $\psi^+_\eps\searrow\psi^+$ Then 
	\begin{align*}
		\int_X e^{-\delta_0\psi_\eps^+}\d V
		\leq \int_Y e^{-\delta_0\psi^+}\d V<\infty.
	\end{align*}
	Denote $h_\eps\coloneqq\frac{\omega'(0)^n}{\d V}$. Note that the density $h_\eps$ is uniformly bounded.
	By the perturbed equation \eqref{eq:eps_IMAF}, one has
	\begin{align}
		I_{\eps,a}(0)
= {} & \int_Y
\left(
V_{\omega_\eps}e^{-u_0+L_\eps(u_0)}
\frac{g_\eps}{h_\eps}
\right)^{-a}
e^{-u_0+L_\eps(u_0)}g_\eps\,\d V \nonumber \\
= {} & V_{\omega_\eps}^{-a}
\int_Ye^{(a-1)u_0-(a-1)L_\eps(u_0)}
h_\eps^a g_\eps^{1-a}\,\d V
\leq K\int_Y e^{-\delta\psi_\eps^++\delta\psi_\eps^-}\,\d V
\leq K. \label{La-density-eps eqn 1}
	\end{align}
On the other hand, note that
\begin{align*}
	\partial_t \log w_\eps 
		=&w_\eps+ \Delta_{\omega'(t)}w_\eps + \dot{L}_\eps(u^\eps)\quad \text{and}\quad
		\p_t\mu_{\omega'(t)}=(w_\eps+\dot L_\eps(u^{\eps}))\mu_{\omega'(t)}.
\end{align*}
Using the above equalities and integration by parts, we compute
\begin{align*}
\frac{\d}{\d t}I_{\eps,a}(t)
= {} & \int_Y\left(-a w_\eps^{-a}\cdot\partial_t\log w_\eps
+w_\eps^{-a}(w_\eps+\dot L_\eps)\right)\mu_{\omega'} \\
= {} & -a\int_Yw_\eps^{-a}\Delta_{\omega'}w_\eps\,\mu_{\omega'}
  -(a-1)\int_Yw_\eps^{1-a}\mu_{\omega'}
  +(1-a)\dot L_\eps(u^{\eps}) I_{\eps,a}(t) \\
\leq {} & -\frac{a}{V_\eps}\int_Yw_\eps^{1-a}\Delta_{\omega'}w_\eps\,{\omega'}^n+(1-a)\dot L_\eps(u^{\eps}) I_{\eps,a}(t) \\
= {} & \frac{a(1-a)}{V_\eps}\int_Yw_\eps^{-a}|\nabla w_\eps|_{\omega'}^2\,\omega'^n +(1-a)\dot L_\eps(u^{\eps}) I_{\eps,a}(t) \\
\leq {} & (1-a)\dot L_\eps(u^{\eps}) I_{\eps,a}(t).
\end{align*}
Combining this with Proposition~\ref{prop:unif_Linfty},
\begin{align*}
	\log\frac{I_{\eps,a}(t)}{I_{\eps,a}(0)}
\leq(a-1)\Big(L_\eps(u_0)-L_\eps(u^\eps(t)) \Big)\leq K.
\end{align*}
Together with \eqref{La-density-eps eqn 1}, we finish the proof.
\end{proof}

\subsection{Long time existence}
\label{subsec:long time existence}
Applying the uniform estimates of the perturbed flow we prove Theorem \ref{thm:longtime_exist_pullback}:

\begin{theorem}
	\label{thm:longtime-exist}
There exists a solution $u\in C_x^{\infty}C_t^1((Y\backslash E)\times(0,\infty))\cap C(Y\times[0,\infty))$ of \eqref{eq:IMAF-normalized} satisfying the following properties: 
\begin{enumerate}
\item[(i)] For any $0<T<\infty$, there are constants $\beta\in(0,1)$ (independent of $T$) and $K_{T}>0$ such that
    \begin{align*}
	|u(x,t)-u(x,t')| \leq K_T|t-t'|^{\beta}
    \quad \text{for $x\in Y$ and $t,t'\in[0,T]$}.
    \end{align*}
\item[(ii)] For any $0<\tau<T<\infty$, there is a constant $K_{\tau,T}>0$ such that
$$
|u(x,t)-u(x,t')| \leq K_{\tau,T}|t-t'|
   \quad \text{for $x\in Y$ and $t,t'\in[\tau,T]$}.
$$
\end{enumerate}
\end{theorem}

\begin{proof}
We first construct the solution of \eqref{eq:IMAF-normalized}. By Arzel\`a-Ascoli theorem and Theorem~\ref{thm:local-C2}, up to a subsequence, we have
\[
u^{\eps} \to u \ \text{in the $C_{x,\mathrm{loc}}^{\infty}C_{t,\mathrm{loc}}^{1}(Y\backslash E\times(0,\infty))$-topology},
\]
for some $u\in C_x^{\infty}C_t^1(Y\backslash E\times(0,\infty))$. Equivalently, for any positive integer $k$, $F\Subset Y\backslash E$ and $0<\tau<T$,
\begin{align*}
	u^{\eps} \to u \ \text{in the $C_x^{k}C_{t}^{1}(F\times[\tau,T])$-topology}.
\end{align*}
In particular, we have
\[
u^{\eps}(x,t) \to u(x,t) \ \text{for any $(x,t)\in Y\backslash E\times(0,\infty)$}.
\]
Since $E$ has zero Lebesgue measure, then by Proposition \ref{prop:unif_Linfty} and the dominated convergence theorem, we obtain
\[
L_\eps(u^\eps) \to -\log\int_{Y\backslash E} e^{-u} g\, \d V = L(u) \quad \text{for any $t\in(0,\infty)$},
\]
and
\begin{align*}
	\frac{1}{V_\omega}(-\dot{u})(\omega+\ddc u)^n
 	=e^{-u+L(u)}g\, \d V \quad \text{ in }Y\backslash E\times(0,\infty).
\end{align*}
The above $u$ is only defined on $Y\backslash E\times(0,\infty)$.

We extend it to $Y\times[0,\infty)$ by showing
\begin{enumerate}
	\item[($*$)] $u(\cdot,t)$ is the limit of a decreasing sequence $\hat u_{\eps_{j}}\in\PSH_{\omega_{\eps_{j}}}$ for each $t>0$. In particular, $u(\cdot,t)\in\PSH_{\pi^* \omega}$ for each $t>0$.
\end{enumerate}
Fix $T>0$ and use the dominated convergence theorem again,
\[
\|u^{\eps}-u\|_{L^{1}(Y\backslash E\times[0,T])} \to 0.
\]
Let $K$ be the constant in Proposition \ref{prop:sub_stab_holder}. We choose $\eps_j>0$ such that
\begin{align*}
	K \|u_{\eps_j} - u\|_{L^1(Y\backslash E\times[0,T])}^\beta < 2^{-j-1}.
\end{align*}
Applying the subadditivity property $(a+b)^\beta \le a^\beta + b^\beta$ for $a,b>0$, one has
\begin{align}
 K \|u_{\eps_j} - u_{\eps_{j+1}}\|_{L^1(Y\backslash E\times[0,T])}^\beta \nonumber \leq & K \left(\|u_{\eps_j} - u\|_{L^1(Y\backslash E\times[0,T])} 
	+ \|u_{\eps_{j+1}} - u\|_{L^1(Y\backslash E\times[0,T])}\right)^\beta \nonumber \\[1mm]
\leq {} &  K \|u_{\eps_j} - u\|_{L^1(Y\backslash E\times[0,T])}^\beta 
	+ K \|u_{\eps_{j+1}} - u\|_{L^1(Y\backslash E\times[0,T])}^\beta \nonumber  \\[1mm]
< {} & 2^{-j-1} + 2^{-j-2} < 2^{-j}.
	\label{eq:subadditivity}
\end{align}
We define
\begin{align*}
	c_{\eps_j} := K \sum_{k=j}^{\infty} \|u_{\eps_k} - u_{\eps_{k+1}}\|_{L^1(Y\times[0,T])}^\beta.
\end{align*}
By \eqref{eq:subadditivity}, we have 
\begin{align*}
	c_{\eps_j} < 2^{-j+1} \rightarrow0, \text{ as }j\rightarrow\infty.
\end{align*}
We claim that $\hat u_{\eps_j}\coloneqq u_{\eps_j}+c_{\eps_j}$ is a decreasing sequence. Indeed, we take $i < j$, which implies $\eps_i > \eps_j$. By \eqref{eq:quasi-decreasing}, one obtains
\begin{align*}
	u_{\eps_j} \le u_{\eps_i} + K \|u_{\eps_i} - u_{\eps_j}\|_{L^1}^\beta.
\end{align*}
Note that
\begin{align*}
	\|u_{\eps_i} - u_{\eps_j}\|_{L^1} \le \sum_{k=i}^{j-1} \|u_{\eps_k} - u_{\eps_{k+1}}\|_{L^1},
\end{align*}
which implies 
\begin{align*}
	K\|u_{\eps_i} - u_{\eps_j}\|_{L^1}^\beta 
	\le K\left( \sum_{k=i}^{j-1} \|u_{\eps_k} - u_{\eps_{k+1}}\|_{L^1} \right)^\beta 
	\le K\sum_{k=i}^{j-1} \|u_{\eps_k} - u_{\eps_{k+1}}\|_{L^1}^\beta
	=c_{\eps_i}-c_{\eps_j}.
\end{align*}
Thus, we have
\begin{align*}
	u_{\eps_j}\leq u_{\eps_i}+(c_{\eps_i}-c_{\eps_j}), \quad\text{ i.e. }\quad
	\hat u_{\eps_j}\leq \hat u_{\eps_i}.
\end{align*}
Denote
\begin{align}
\label{eq:def-hatu}
	\hat u(x,t) \coloneqq \lim_{j\to\infty}\hat u_{\eps_{j}}(x,t) \ \ \text{for $(x,t)\in Y\times[0,T]$}.
\end{align}
Then $\hat u(\cdot,t)$ is the limit of a decreasing sequence $\hat u_{\eps_{j}}\in\PSH_{\eps_{j}}$ for each $t\in[0,T]$.
On the other hand, for any $(x,t)\in Y\backslash E\times[0,T]$, we have
\begin{align}
\label{eq:limit-u_eps_j}
	u(x,t) = \lim_{j\to\infty}u_{\eps_{j}}(x,t) = \lim_{j\to\infty}(\hat u_{\eps_{j}}(x,t)-c_{\eps_{j}}) = \hat u(x,t).
\end{align}
This shows that for any $t\in[0,T]$, $u(\cdot,t)$ can be extended to a global function on $Y$ satisfying the required decreasing approximation property. Since $T$ is arbitrary, then ($*$) follows.

\medskip

Next we show $u$ satisfies (i) and (ii).
On the other hand, for any $x\in Y$, apply Proposition~\ref{prop:sub_stab_holder} to $u_{\eps_{j}}$,
\begin{align*}
|u_{\eps_j}(x,t)-u_{\eps_j}(x,t')|
\leq K|t - t'|^\beta.
\end{align*}
Taking $j\rightarrow\infty$, by \eqref{eq:limit-u_eps_j}, we obtain (i). By the similar argument and Proposition \ref{prop:local_t-deriv_upper}, we obtain (ii).

\medskip

Finally, we prove $u\in C(Y\times[0,\infty))$. From (i), $t\mapsto u_t$ is $t$-H\"older. It suffices to show that for any $t>0$, $u_t\in C^0(X)$. Denote \(w(t):=-\dot u_t>0\) and rewrite the equation \eqref{eq:IMAF-normalized}
	\begin{align}
	\label{eq:iMAf-ellip-version}
		(\omega+\ddc u_t)^n=V_\omega e^{-u_t+L(u_t)}\frac{g}{w(t)}\,\d V
		\eqqcolon h(t)\,\d V \quad \text{ on }Y\backslash E.
	\end{align}
    we extend $h(t)$ by $0$ cross $E$. Since $u_t\in L^\infty(Y)$, by Bedford--Taylor theory \cite{BT76,BT82}, $\omega_{u_t}^n$ put no mass on $E$. Thus, the equation \eqref{eq:iMAf-ellip-version} can be viewed as an equation of measures. Fix $a>1$ in Lemma \ref{lem:La-density-eps} and denote 
	\begin{align*}
		r_*\coloneqq \frac{ap_{0}}{a+p_{0}-1}>1.
	\end{align*}
We claim that for any  \(r\in(1,r_*)\),
	\begin{align}
	\label{claim:Lr-density}
		\sup_{0<t\leq T}\left\|\frac{g}{w(t)}\right\|_{L^r(Y\backslash E,\d V)}
\leq K_{r,T}.
	\end{align}
Given this claim, we have $h(t)\in L^r(\d V)$. By \cite[Theorem~A]{EGZ09}, we obtain $u_t\in C^0(Y)$. Next we prove the claim. By Lemma~\ref{lem:La-density-eps} and Proposition~\ref{prop:unif_Linfty}, one has
	\begin{align*}
		\int_Yw_\eps(t)^{-a}g_\eps\,\d V
\leq KI_{\eps,a}(t)\leq K.
	\end{align*}
Let $F_j\Subset Y\backslash E$ be the exhaustion subset. Recall that $w_\eps=-\dot u^\eps\rightarrow-\dot u=w$ locally uniformly on \(Y\backslash E\). Then
\begin{align*}
		\int_{F_j}w(t)^{-a}g \,\d V
\leq\liminf_{\eps\to0}
\int_{F_j}w_\eps(t)^{-a}g_\eps \, \d V
\leq K_{a,T}.
	\end{align*}
	Letting \(j\to\infty\) and using monotone convergence yields
	\begin{align} \label{lem:La-density-w(t)}
		\int_{Y\backslash E}w(t)^{-a}g \, \d V\leq K_{a,T}.
	\end{align}
One writes
	\begin{align*}
		\left(\frac gw\right)^r
       =(w^{-a}g)^{\frac{r}a}g^{r(1-\frac1a)}.
	\end{align*}
	By \eqref{lem:La-density-w(t)} and the H\"older's inequality with conjugate exponents \(\frac{a}r\) and \(\frac{a}{a-r}\), one  yields
	\begin{align*}
		\int_{Y\backslash E}\left(\frac gw\right)^r \d V
\leq
\left(\int_{Y\backslash E}w^{-a}g\,\d V\right)^{\frac{r}{a}}
\left(\int_{Y\backslash E}g^q\,\d V\right)^{\frac{a-r}a}
\leq K_{r,a,T} \left(\int_Yg^q\,\d V\right)^{\frac{a-r}a},
	\end{align*}
	where $q=\frac{r(a-1)}{a-r}>0$. Note that $r<r_*$ is equivalent to \(q<p\).
	Since $g\in L^p(\d V)$, then we obtain the claim \eqref{claim:Lr-density}.
\end{proof}

\section{Convergence and uniqueness of the inverse MA flow}
\label{sec:Convergence of inverse MA flow}

In this section, we prove the uniqueness and convergence part of Theorem \ref{thm: main_theorem}
(see Theorem \ref{thm:d_1_conv_inverse_MA} and Corollary \ref{cor:uniqueness-iMAf} below). To this end, we first
show that the inverse Monge--Amp\`ere flow, whose existence was proved in the
previous sections, coincides with its minimizing movement version, constructed
using the general theory developed in \cite{AGS,May}; see also
\cite{Streets,BDL17} for the corresponding construction of the Calabi flow.
Before recalling the minimizing movement construction, we review the
$d_2$-continuity and compactness properties of the Ding functional that carry
over readily from the smooth case.

Recall that
\begin{align*}
	D(u)=L(u)-E(u)
	=-\log\int_X e^{-u}\mu-E(u), \ u \in \mathcal E^1(X,\omega).
\end{align*}

\begin{lemma}
\label{lem:D-compactness}
(i) $ D: \mathcal E^2(X,\omega) \to \mathbb{R} $ is $d_2$-continuous and convex.\\
(ii) If $\{u_j\}_j \subset \mathcal E^2(X,\omega)$ is a $d_2$-bounded sequence and $\liminf_j D(u_j) = \inf_{\mathcal E^2(X,\omega)} D$, then there exist a subsequence $\{u_{j_k}\}_k$ and a KE potential $u\in\PSH_\omega \cap C(X) \cap C^\infty(X_\reg)$ such that $d_1(u_{j_k},u)\to 0$.
\end{lemma}
\begin{proof} By \cite{DNG18}, the functional $E$ is $d_2$-continuous. Moreover, $d_2$-convergence implies $d_1$-convergence, while $L$ is $d_1$-continuous by \cite[Lemma 4.1]{BBEGZ19}. Hence $L$ is $d_2$-continuous, proving (i). Convexity is a consequence of \cite{Bern15}.

With regard to (ii), it follows from \cite[Theorem 4.8]{BBEGZ19} that a minimizer $u \in \mathcal E^2(X,\omega) \subset \mathcal E^1(X,\omega)$ must be a KE potential.

Since $d_2$-boundedness implies $d_1$-boundedness, the rest of the proof of (ii) follows by the same argument as in \cite[Theorem~4.18]{Dar19}.
\end{proof}

Following the minimizing movement construction of \cite{Streets} building on \cite{May} (cf.\ \cite[Section~2.5]{BDL17}) for the Calabi flow, we can define the minimizing movement version of the inverse MA flow as follows. We
fix $\check{u}_0\in\cE^2(X,\omega)$. For $t\geq0$ and a
positive integer $m$, set $\check{u}_t^{m,0}=\check{u}_0$. Given $\check{u}_t^{m,j}$, define
$\check{u}_t^{m,j+1}$ to be a minimizer on $\cE^2(X,\omega)$ of
\begin{align}
\label{eq:minimizing-movement-Ding}
	v\longmapsto \frac{1}{2}d_2^2\big(v,\check{u}_t^{m,j}\big)
	+\frac{t}{m}D(v),
\end{align}
for $j=0,\ldots,m-1$, and set $\check{u}_t^m\coloneqq \check{u}_t^{m,m}$. Provided that
these minimizers exist and $\check{u}_t^m$ converges as $m\to\infty$, the $d_2$-limit
\begin{align*}
	\check{u}_t\coloneqq\lim_{m\to\infty}\check{u}_t^m \in \mathcal E^2(X,\omega)
\end{align*}
is called the \textit{minimizing movement inverse MA flow} starting from $\check{u}_0$.
We also refer to $t\mapsto \check{u}_t$ as the weak inverse MA flow. By \cite[Theorem~4.0.4]{AGS}, it is characterized
as the unique solution of the \textit{evolution variational inequality} (EVI)
\begin{align}
\label{eq:evolution_vari_ineq}
	\frac{1}{2}\frac{\d}{\d t}d_2^2(\check{u}_t,v)
	\leq D(v)-D(\check{u}_t),
\end{align}
for almost every $t>0$ and all $v\in\cE^2$.
In addition, the curve $t\mapsto \check{u}_t$ can be viewed as a curve of steepest descent with respect to $D$ in the sense that
\begin{align}
\label{eq:d_D(u_t)}
	-\frac{\d}{\d t}D(\check{u}_t)=|(\partial D)(\check{u}_t)|\left|\frac{\d \check{u}_t}{\d t} \right|=\left|\frac{\d \check{u}_t}{\d t} \right|^2\geq 0
\end{align}
for almost every $t$, where
\begin{flalign*}
\left|(\partial D)(y)\right| \coloneqq
\limsup_{z\to y}
\frac{\bigl(D(y)-D(z)\bigr)^{+}}{d_2(y,z)}, \ \ \ \ \ \left|\frac{\d \check{u}_t}{\d t}\right| \coloneqq
\lim_{s\to t}
\frac{d_2(\check{u}_s,\check{u}_t)}{|s-t|}.
\end{flalign*}

\begin{proposition}\label{prop: inverse_MA_equal_to_min_movement}
Given any initial $u_0\in\cH_\omega$, every solution $t\mapsto u_t$ of the inverse MA flow in the sense of
\textup{Definition~\ref{definition of IMAF}} coincides with the corresponding minimizing movement flow $t\mapsto\check u_t$.
\end{proposition}

\begin{proof}
Let $t\mapsto u_t$ be a solution in the sense of
Definition~\ref{definition of IMAF}. Throughout this proof we use the
unnormalized equation \eqref{eq:iMAf-X_reg}. Fix
$0<\tau<T<\infty$. On $X \times [\tau,T]$, the potential $u$ is uniformly bounded
and $t$-Lipschitz. $u$ is also smooth in space and $C^1$ in time on $X_\reg\times [\tau,T]$. 

We first compute the variation of  $D$. Taking derivatives of $E$ and $L$ along the inverse MA flow we obtain that
\begin{equation}\label{eq: Energy_deriv}
\begin{aligned}
\frac{\d}{\d t}E(u_t)
&=\frac{1}{V_\omega}\int_X\dot u_t\,\omega_{u_t}^n,
\\
\frac{\d}{\d t}L(u_t)
&=\int_X\dot u_t e^{-u_t+L(u_t)}\mu,\qquad t \in [\tau, T].
\end{aligned}
\end{equation}
Indeed, both formulas follow, since $u$ is  $t$-Lipschitz on $X \times [\tau, T]$ we can to apply the dominated convergence theorem together with standard convergence arguments of Bedford--Taylor theory.

Using
\eqref{eq:iMAf-X_reg}, we obtain
\begin{align}
\label{eq: IMAF_D_dissipation}
-\frac{\d}{\d t}D(u_t)
=\frac{1}{V_\omega}\int_X\dot u_t^{\,2}\,\omega_{u_t}^n.
\end{align}
We next bound the metric speed. The path-length upper bound for the
affine segment joining $u_t$ to $u_{t+h}$ gives
\begin{align*}
d_2(u_t,u_{t+h})
\leq\int_0^1\left(
\frac{1}{V_\omega}\int_X|u_{t+h}-u_t|^2
\omega_{(1-s)u_t+s u_{t+h}}^n
\right)^{1/2}\d s.
\end{align*}
This bound follows from 
\cite[Proposition~1.16]{DNG18}. In particular,
$d_2(u_t,u_{t+h})\leq\|u_{t+h}-u_t\|_{L^\infty}$, so the curve $t \to u_t$
locally $d_2$-Lipschitz and is $d_2$-continuous at $t=0$.
Dividing both sides by $|h|$ and using the dominated convergence theorem  again together with standard convergence arguments of Bedford--Taylor theory gives
\begin{align}
\label{eq: IMAF_metr_speed}
\left|\frac{\d u_t}{\d t}\right|^2
\leq \frac{1}{V_\omega}\int_X\dot u_t^{\,2}\,\omega_{u_t}^n
\end{align}
for almost every $t>0$.

We fix $t>0$ and
$v\in\cH_\omega$, and let $\gamma_{t,s}$, $0\leq s\leq1$, be the bounded weak geodesic from $u_t$ to $v$. We will show that its initial velocity satisfies
\begin{equation}\label{eq: geod_initial_bound}
\frac{1}{V_\omega}\int_X
\left|\left.\frac{\partial}{\partial s}\right|_{s=0^+}
\gamma_{t,s}\right|^2\omega_{u_t}^n
\leq d_2^2(u_t,v).
\end{equation}
Indeed, on the resolution $\pi: Y \to X$ we fix $K\Subset Y\setminus E$. We choose smooth functions $f^j\geq u$
on $X \times [0 \times \infty)$ converging uniformly to $u_t$, with $f^j=u_t$ on a neighborhood of $K$.
The $\pi^*\omega$-psh envelopes
$u^j_t:=P_{\pi^*\omega}(f^j_t)$ satisfy $u_t\leq u^j_t\leq f^j_t$ and equal
$u_t$ near $K$. Let $\gamma^j_{t,s}, \ s \in [0,1]$ be the weak geodesics joining $u^j_t$ to $v_t$.

 The comparison principle for weak geodesics gives $\gamma^j_{t,s}\geq\gamma_{t,s}$ and uniform
convergence on $X \times [0,1]$ as $j \to \infty$. On $K$, the initial values agree, so
comparison and convexity in $s$ give
\begin{align*}
\left.\frac{\partial}{\partial s}\right|_{s=0^+}\gamma_{t,s}(x)
\leq
\left.\frac{\partial}{\partial s}\right|_{s=0^+}\gamma^j_{t,s}(x)
\leq\frac{\gamma^j_{t,s}(x)-u_t(x)}{s},\qquad s>0, \ \ x \in K.
\end{align*}
 Letting first $j\to\infty$ and then $s\searrow0$ we obtain that $\frac{\partial}{\partial s}\big|_{s=0^+}\gamma^j_{t,s} \to \frac{\partial}{\partial s}\big|_{s=0^+}\gamma_{t,s}$ on $K$.

Since $f^j$ is smooth and $v\in\cH_\omega$, the of formula
\cite[Theorem~1.21]{DNG18} applies to the weak geodesic $\gamma^j_{t,s}$ and gives
\begin{align*}
\frac{1}{V_\omega}\int_Y
\left|\left.\frac{\partial}{\partial s}\right|_{s=0^+}
\gamma^j_{t,s}\right|^2\omega_{u^j}^n
=d_2^2(u^j,v).
\end{align*}
 Moreover, $u^j_t=u_t$ on a neighborhood of $K$, so locality of the
Monge--Amp\`ere operator gives $\omega_{u^j}^n=\omega_{u_t}^n$ there.
We can therefore apply Fatou's lemma on $K$ with the fixed measure
$\omega_{u_t}^n$, using the pointwise convergence of the initial
 velocities proved above:
\begin{align*}
\frac{1}{V_\omega}\int_K
\left|\left.\frac{\partial}{\partial s}\right|_{s=0^+}
\gamma_{t,s}\right|^2\omega_{u_t}^n&\leq\liminf_{j\to\infty}\frac{1}{V_\omega}\int_K
\left|\left.\frac{\partial}{\partial s}\right|_{s=0^+}
\gamma^j_{t,s}\right|^2\omega_{u_t}^n\\
&=\liminf_{j\to\infty}\frac{1}{V_\omega}\int_K
\left|\left.\frac{\partial}{\partial s}\right|_{s=0^+}
\gamma^j_{t,s}\right|^2\omega_{u^j_t}^n \leq\liminf_{j\to\infty}d_2^2(u^j_t,v)
=d_2^2(u_t,v).
\end{align*}
The last equality follows from uniform convergence $u^j_t\to u_t$, which
implies $d_2(u^j_t,u_t)\to0$, and the triangle inequality.
Since $K \subset Y\setminus E$ was arbitary \eqref{eq: geod_initial_bound} follows.

By the convexity of $D$, the affinity of $E$ along $\gamma_{t,s}$, and the
concavity of $E$ along affine segments, for any $t>0$ we  have
\begin{align*}
D(v)-D(u_t)
&\geq\left.\frac{\d}{\d s}\right|_{s=0^+}D(\gamma_{t,s})\\
&\geq\int_X
\left.\frac{\partial}{\partial s}\right|_{s=0^+}\gamma_{t,s}
\left(e^{-u_t+L(u_t)}\mu-\frac{1}{V_\omega}\omega_{u_t}^n\right)\\
&=-\frac{1}{V_\omega}\int_X
\left.\frac{\partial}{\partial s}\right|_{s=0^+}\gamma_{t,s}
\dot u_t\,\omega_{u_t}^n\\
&\geq-d_2(u_t,v)
\left(\frac{1}{V_\omega}\int_X\dot u_t^{\,2}\,\omega_{u_t}^n\right)^{1/2},
\end{align*}
where the second line  is justified by $E(\gamma_{t,s})-E(u_t)
\le \frac1{V_\omega}\int_X(\gamma_{t,s}-u_t)\,\omega_{u_t}^n$ and convexity, the thid line by definition of the inverse MA flow and the fourth line by the Cauchy--Schwarz inequality. 
    
By the $d_2$-density of $\cH_\omega$ and
Lemma~\ref{lem:D-compactness}(i), the resulting inequality holds for
every $v\in\cE^2$. Taking the limsup as $v\to u_t$ gives
\begin{align}
\label{eq:IMAF-slope-bound}
|(\partial D)(u_t)|^2
\leq \frac{1}{V_\omega}\int_X\dot u_t^{\,2}\,\omega_{u_t}^n.
\end{align}

Adding \eqref{eq: IMAF_metr_speed} and \eqref{eq:IMAF-slope-bound},
dividing by two, and using \eqref{eq: IMAF_D_dissipation}, we obtain,
for almost every $t\in[\tau,T]$,
\begin{align*}
\frac12\left(\left|\frac{\d u_t}{\d t}\right|^2
+|(\partial D)(u_t)|^2\right)\leq\frac{1}{V_\omega}\int_X\dot u_t^{\,2}\,\omega_{u_t}^n=-\frac{\d}{\d t}D(u_t).
\end{align*}
To integrate this inequality, note that both $E$ and $L$ are
$1$-Lipschitz with respect to the uniform norm on bounded potentials.
Since $t\mapsto u_t$ is Lipschitz in that norm on $[\tau,T]$, the map
$t\mapsto D(u_t)=L(u_t)-E(u_t)$ is Lipschitz, and hence absolutely
continuous, on this interval. The fundamental theorem of calculus
therefore gives
\begin{align}\label{eq: weak_EVI}
\frac12\int_\tau^T
\left(\left|\frac{\d u_t}{\d t}\right|^2
+|(\partial D)(u_t)|^2\right)\d t\leq-\int_\tau^T\frac{\d}{\d t}D(u_t)\,\d t=D(u_\tau)-D(u_T).
\end{align}

This inequality is weaker than the EVI inequality \eqref{eq:evolution_vari_ineq}. That being said, since the curve $t \to \check u_t$ solves the EVI inequality, \eqref{eq: weak_EVI} immediately implies using 
\cite[Theorem~4.2]{MS20} that $u_t = \check u_t$, concluding the argument.
\end{proof}

Since the minimizing movement MA flow is unique, we immediately obtain our main uniqueness result:

\begin{corollary}[Uniqueness]
\label{cor:uniqueness-iMAf}
For each $u_0\in\cH_\omega$, there is at most one solution
of the inverse MA flow starting from $u_0$ in the sense of
\textup{Definition~\ref{definition of IMAF}}.
\end{corollary}

Next we obtain our main long time convergence result:

\begin{theorem}
\label{thm:d_1_conv_inverse_MA} Let $u_0 \in \mathcal H_\omega$ and let $t \to u_t $ be a solution of the inverse MA flow in the sense of
\textup{Definition~\ref{definition of IMAF}}. Then $X$ admits a KE metric if and only if there exists $u_\infty \in \PSH_\omega \cap C(X)$  such that $u_t \to u_\infty$  in $L^1(X)$ as $t \to \infty$. In this latter case $\omega_{u_\infty}$ is a KE metric that is smooth on $X_\reg$.
\end{theorem}
\begin{proof}
By Proposition~\ref{prop: inverse_MA_equal_to_min_movement}, we have
$u_t=\check{u}_t$ for all $t\geq0$. For any $v\in\cE^2(X,\omega)$, integrating
\eqref{eq:evolution_vari_ineq} from $0$ to $t$ and using that
$s\mapsto D(u_s)$ is non-increasing gives
\begin{align*}
	t\bigl(D(u_t)-D(v)\bigr)
	&\leq \frac{1}{2}\left(d_2^2(u_0,v)-d_2^2(u_t,v)\right)
	\leq \frac{1}{2}d_2^2(u_0,v).
\end{align*}
Consequently,
\begin{align*}
	\lim_{t\to\infty}D(u_t)=\inf_{\cE^2}D=\inf_{\cE^1}D,
\end{align*}
where the infimum is allowed to be $-\infty$. The two infima agree because,
for every $w\in\cE^1$, the bounded truncations $w_j:=\max\{w,-j\}\in\cE^2$
satisfy $E(w_j)\to E(w)$ and $L(w_j)\to L(w)$, hence $D(w_j)\to D(w)$.

First suppose that $X$ admits a KE metric. Its continuous potential belongs
to $\cE^2$ and minimizes $D$ by \cite[Theorem~4.8]{BBEGZ19}.
By \cite[Theorem~1.5]{Bacak13}, there is a minimizer $u_\infty$ of $D$ such
that $u_t$ converges $d_2$-weakly to $u_\infty$ as $t\to\infty$.
By \cite[Theorem~4.8 and Lemma~3.6]{BBEGZ19}, $u_\infty$ is a KE potential
in $\PSH_\omega\cap C(X)\cap C^\infty(X_\reg)$.

By the $d_2$-contractivity of the gradient-flow semigroup
\cite[Theorem~4.0.4]{AGS}, applied to $t\mapsto u_t$ and the stationary
flow $t\mapsto u_\infty$, the map
$t\mapsto d_2(u_t,u_\infty)$ is non-increasing. In particular, the
trajectory $u_t$ is $d_2$-bounded.

Let $t_k\to\infty$ be arbitrary. Since $D(u_t)\to\inf_{\cE^2}D$,
Lemma~\ref{lem:D-compactness}(ii) gives, after passing to a subsequence,
that $u_{t_k}$ converges in $d_1$ to a minimizer $v_\infty$ of $D$.
By \cite[Theorem~5.3]{BDL17}, it also converges $d_2$-weakly to $v_\infty$. Although \cite{BDL17} treats only the smooth K\"ahler setting, the proof of this fact (and many other statements) carries over verbatim to the present singular setting.

Since the whole trajectory converges $d_2$-weakly to $u_\infty$, we have
$v_\infty=u_\infty$. Thus every sequence $t_k\to\infty$ has a subsequence
converging to $u_\infty$ in $d_1$, which proves that
$d_1(u_t,u_\infty)\to0$. In particular, $u_t\to u_\infty$ in $L^1(X)$.

Conversely, suppose that $u_t\to u_\infty$ in $L^1(X)$ for some
$u_\infty\in\PSH_\omega\cap C(X)$. Since the adapted measure $\mu$ is `tame',
\cite[Proposition~1.4]{BBEGZ19} gives
$e^{-u_{t_j}}\to e^{-u_\infty}$ in $L^1(\mu)$ (c.f. Theorem 1.1 in \cite{DZ26}). Thus
$L(u_t)\to L(u_\infty)$ as $t\to\infty$.
The $L^1$-upper semicontinuity of $E$ now yields
\begin{align*}
	D(u_\infty)
	&=L(u_\infty)-E(u_\infty)
	\leq\liminf_{t\to\infty}\bigl(L(u_t)-E(u_t)\bigr)
	=\inf_{\cE^1}D.
\end{align*}
Thus $u_\infty\in\cE^1$ minimizes $D$. By \cite[Theorem~4.8 and Lemma~3.6]{BBEGZ19}, $\omega_{u_\infty}$ is a KE
metric that is smooth on $X_\reg$.
\end{proof}

\begingroup 
\renewcommand\refname{References}
	\bibliographystyle{alpha}
	\bibliography{Reference_IMAF.bib}

\endgroup 
\end{document}